\documentclass[11pt, a4paper, oneside, reqno]{amsart}

\usepackage[usenames, dvipsnames]{color}
\definecolor{darkblue}{rgb}{0.0, 0.0, 0.45}
\definecolor{lightblue}{RGB}{240,248,255}
\definecolor{lightblue2}{rgb}{0.68, 0.85, 0.9}
\definecolor{lightcyan}{rgb}{0.88, 1.0, 1.0}
\definecolor{palepink}{rgb}{0.98, 0.85, 0.87}

\usepackage[colorlinks	= true,
raiselinks	= true,
linkcolor	= darkblue, %MidnightBlue,
citecolor	= Mahogany,
urlcolor	= ForestGreen,
pdfauthor	= {Robert D. McAllister},
pdftitle	= {},
pdfkeywords	= {},
pdfsubject	= {},
plainpages	= false]{hyperref}

\usepackage{dsfont,amssymb,amsmath,enumitem} %,etaremune,savesym}
\usepackage{amsfonts,dsfont,mathtools, mathrsfs,amsthm} 
\usepackage[amssymb, thickqspace]{SIunits}
\usepackage{fancyhdr,mdframed,nicefrac}

\usepackage[norelsize, ruled, vlined]{algorithm2e}
\usepackage{algcompatible}

\usepackage{epsfig}
\usepackage{graphicx}
\usepackage{float}
\usepackage{caption}
\usepackage{subcaption}

\usepackage{courier}

\usepackage{multirow}
\usepackage{bigstrut}

\allowdisplaybreaks
\date{\today}
\makeatletter
\def\@settitle{\begin{center}%
		\baselineskip14\p@\relax
		\normalfont\LARGE\scshape\bfseries
		\@title
	\end{center}%
}

\def\@setauthors{%
  \begingroup
  \def\thanks{\protect\thanks@warning}%
  \trivlist
  \centering\footnotesize \@topsep30\p@\relax
  \advance\@topsep by -\baselineskip
  \item\relax
  \author@andify\authors
  \def\\{\protect\linebreak}%
  \authors%
  \ifx\@empty\contribs
  \else
    ,\penalty-3 \space \@setcontribs
    \@closetoccontribs
  \fi
  \endtrivlist
  \endgroup
}

\makeatother
\makeatletter

\def\subsection{\@startsection{subsection}{2}%
	\z@{.5\linespacing\@plus.7\linespacing}{.5\linespacing}%
	{\normalfont\large\bfseries}}

\def\subsubsection{\@startsection{subsubsection}{3}%
	\z@{.5\linespacing\@plus.7\linespacing}{.5\linespacing}%
	{\normalfont\itshape}}

\usepackage{multirow}
\usepackage{bigstrut}
\usepackage{wrapfig}
\usepackage{caption}

\usepackage[numbers,sort&compress]{natbib}

\usepackage[noabbrev,capitalize]{cleveref}

\usepackage{makecell}

\newcommand{\bbR}{\mathbb{R}}
\newcommand{\bbI}{\mathbb{I}}

\newcommand{\bbX}{\mathbb{X}}
\newcommand{\bbU}{\mathbb{U}}
\newcommand{\bbZ}{\mathbb{Z}}
\newcommand{\bbW}{\mathbb{W}}

\newcommand{\bbB}{\mathbb{B}}

\newcommand{\dsP}{\mathds{P}}
\newcommand{\dsQ}{\mathds{Q}}

\newcommand{\bfx}{\mathbf{x}}
\newcommand{\bfu}{\mathbf{u}}
\newcommand{\bfw}{\mathbf{w}}
\newcommand{\bfv}{\mathbf{v}}

\newcommand{\bfM}{\mathbf{M}}
\newcommand{\bfA}{\mathbf{A}}
\newcommand{\bfB}{\mathbf{B}}

\newcommand{\bfG}{\mathbf{G}}

\newcommand{\bfK}{\mathbf{K}}
\newcommand{\bfP}{\mathbf{P}}

\newcommand{\bfSigma}{\boldsymbol{\Sigma}}

\newcommand{\bfS}{\mathbf{S}}

\newcommand{\calZ}{\mathcal{Z}}

\newcommand{\calX}{\mathcal{X}}
\newcommand{\calM}{\mathcal{M}}

\newcommand{\calD}{\mathcal{D}}
\newcommand{\calP}{\mathcal{P}}
\newcommand{\calB}{\mathcal{B}}

\newcommand{\calQ}{\mathcal{Q}}

\newcommand{\calJ}{\mathcal{J}}

\newcommand{\calK}{\mathcal{K}}

\newcommand{\ev}{\mathbb{E}}

\newcommand{\tr}{\textnormal{tr}}

\newtheorem{theorem}{Theorem}[section]
\newtheorem{corollary}[theorem]{Corollary}

\newtheorem{lemma}{Lemma}[section]
\newtheorem{proposition}[lemma]{Proposition}

\theoremstyle{definition}
\newtheorem{definition}{Definition}[section]
\newtheorem{remark}{Remark}[section]

\newtheorem{assumption}{Assumption}[section]

\crefname{equation}{}{}
\crefname{Equation}{}{}
\crefname{assumption}{Assumption}{Assumptions}
\crefname{definition}{Definition}{Definitions}
\crefname{section}{Section}{Sections}
\crefname{conjecture}{Conjecture}{Conjectures}

\title[Distributionally Robust Model Predictive Control]{Distributionally Robust Model Predictive Control: \\ Closed-loop Guarantees and Scalable Algorithms}
\author{Robert D. McAllister and Peyman Mohajerin Esfahani}
\thanks{The authors are with the Delft Center for Systems and Control, Delft University of Technology, Delft, The Netherlands. Emails: \href{mailto:R.D.Mcallister@tudelft.nl}{R.D.Mcallister@tudelft.nl}, \href{mailto:P.MohajerinEsfahani@tudelft.nl}{P.MohajerinEsfahani@tudelft.nl}. This research is supported by the European Research Council (ERC) under grant TRUST-949796.}

\begin{document}

\begin{abstract}
    We establish a collection of closed-loop guarantees and propose a scalable optimization algorithm for distributionally robust model predictive control (DRMPC) applied to linear systems, convex constraints, and quadratic costs. Via standard assumptions for the terminal cost and constraint, we establish distribtionally robust long-term and stage-wise performance guarantees for the closed-loop system. We further demonstrate that a common choice of the terminal cost, i.e., via the discrete-algebraic Riccati equation, renders the origin input-to-state stable for the closed-loop system. This choice also ensures that the exact long-term performance of the closed-loop system is \emph{independent} of the choice of ambiguity set for the DRMPC formulation. Thus, we establish conditions under which DRMPC does \emph{not} provide a long-term performance benefit relative to stochastic MPC. To solve the DRMPC optimization problem, we propose a Newton-type algorithm that empirically achieves \textit{superlinear} convergence and guarantees the feasibility of each iterate. We demonstrate the implications of the closed-loop guarantees and the scalability of the proposed algorithm via two examples. To facilitate the reproducibility of the results, we also provide open-source code to implement the proposed algorithm and generate the figures.
\end{abstract}

\maketitle


\section{Introduction}

Model predictive control (MPC) defines an implicit control law via a finite horizon optimal control problem. This optimal control problem is defined by the stage cost $\ell(x,u)$, state/input constraints, and a linear discrete-time dynamical model
\begin{equation*}
    x^+ = Ax+Bu+Gw, 
\end{equation*}
in which $x$ is the state, $u$ is the manipulated input, and $w$ is the disturbance. 
The primary difference between variants of MPC (e.g., nominal, robust, and stochastic MPC) is their approach to modeling the disturbance $w$ in the optimization problem. 

In nominal MPC, the optimization problem uses a nominal dynamical model, i.e., $w=0$. Nonetheless, feedback affords nominal MPC a nonzero margin of \textit{inherent} robustness to disturbances \citep{grimm:messina:tuna:teel:2004,yu:reble:chen:allgower:2014,
allan:bates:risbeck:rawlings:2017}. This nonzero margin, however, may be insufficient in certain safety-critical applications with high uncertainty. Robust MPC (RMPC) and stochastic MPC (SMPC) offer a potential means to improve on the inherent robustness of nominal MPC by characterizing the disturbance and including this information in the optimal control problem. 

RMPC describes the disturbance via a set $\bbW$ and requires that the state and input constraints in the optimization problem are satisfied for all possible realizations of $w\in\bbW$. The objective function of RMPC considers only the nominal system ($w=0$) and these methods are sometimes called tube-based MPC if the constraint tightening is computed offline \citep{mayne:seron:rakovic:2005,goulart:kerrigan:maciejowski:2006}.

SMPC includes a stochastic description of the disturbance $w\sim \dsP$ ($w$ is distributed according to the probability distribution $\dsP$) and defines the objective function based on the expected value of the cost function subject to this distribution \citep{cannon:kouvaritakis:wu:2009,farina:giulioni:scattolini:2016,mesbah:2016,lorenzen:dabbene:tempo:allgower:2016}. The performance of SMPC therefore depends on the disturbance distribution $\dsP$. This stochastic description of the disturbance also permits the use of so-called chance constraints. While SMPC refers to a range of methods, characterized by their use of a distribution in the optimization problem, we focus specifically on SMPC for linear systems, quadratic costs, robust constraints, and with expected value objective functions. We do not consider chance constraints. Analogous to nominal MPC, feedback affords SMPC a small margin of inherent distributional robustness, i.e., robustness to inaccuracies in the disturbance distribution \citep{mcallister:rawlings:2023}. If this distribution is identified from limited data, however, there may be significant distributional uncertainty. 
Therefore, a distributionally robust (DR) approach to the SMPC optimization problem may provide desirable benefits in applications with high uncertainty and limited data. 

Advances in distributionally robust optimization (DRO) have inspired a range of distributionally robust MPC (DRMPC) formulations. In general, these problems take the form
\begin{equation}
    \min_{\theta\in\Pi(x)} \max_{\dsP\in\calP}\ev_{\dsP}[J(x,\theta,\bfw)],
    \label{eq:dro_intro} 
\end{equation}
in which $x$ is the current state of the system, $\theta$ defines the control inputs for the MPC horizon (potentially as parameters in a previously defined feedback law), $\ev_{\dsP}\left[\cdot\right]$ denotes the expected value with respect to the distribution $\dsP$, and $\calP$ is the ambiguity set for the distribution $\dsP$ of the disturbances $\bfw$. The goal is to select $\theta$ to minimize the worst-case expected value of the cost function $J(\cdot)$ and satisfy the constraints embedded in the set $\Pi(x)$. Note that SMPC and RMPC are special cases of DRMPC via specific choices of $\calP$. 

The key feature of all MPC formulations is that the finite horizon optimal control problem in \cref{eq:dro_intro} is solved with an updated state estimate at each time step, i.e., a rolling horizon approach. With this approach, DRMPC defines an implicit feedback control law $\kappa(x)$ and the closed-loop system

\begin{equation}\label{eq:fcl}
    x^+ = Ax+B\kappa(x)+Gw.
\end{equation}

The performance of this controller is ultimately defined by this closed-loop system and the stage cost. In particular, we often define performance based on the expected average closed-loop stage cost at time $k\geq 1$, i.e., 
\begin{equation*}
    \calJ_k(\dsP) := \ev_{\dsP}\left[\frac{1}{k}\sum_{i=0}^{k-1}\ell\Big(\phi(i),\kappa(\phi(i))\Big)\right],
\end{equation*}
in which $\phi(i)$ is the closed-loop state trajectory defined by \cref{eq:fcl} and $\dsP$ is the distribution for the closed-loop disturbance.
 
In this work, we focus on DRMPC formulations for linear systems with additive disturbances and quadratic costs. We note that there are also DRMPC formulations that consider parameteric uncertainty  \citep{coulson:lygeros:dorfler:2021} and piecewise affine cost functions are also considered in \citep{micheli:summers:lygeros:2022}. In both cases, the proposed DRMPC formulation solves for only a single input trajectory for all realizations of the disturbance.  

To better address the realization of uncertainty in the open-loop trajectory, RMPC/SMPC formulations typically solve for a trajectory of parameterized control policies instead of a single input trajectory. 
A common choice of this parameterization is the state-feedback law $u=Kx+v$ in which $K$ is the fixed feedback gain and the parameter to be optimized is $v$. Using this parameterization, several DRMPC formulations were proposed to tighten probabilistic constraints for linear systems based on different ambiguity sets \citep{mark:liu:2020,tan:yang:chen:li:2022,fochesato:lygeros:2022,li:tan:wu:duan:2021}. In these formulations, however, the cost function is unaltered from the corresponding SMPC formulation due to the fixed feedback gain in the control law parameterization.

If the control law parameterization is chosen as a more flexible feedback affine policy (see \cref{eq:feedbackaffine}), first proposed for MPC formulations in \citep{goulart:kerrigan:maciejowski:2006}, distributional uncertainty in the cost function is nontrivial to the DRMPC problem. \citet{vanparys:kuhn:goulart:morari:2015} propose a tractable method to solve linear quadratic control problems with unconstrained inputs and a distributionally robust chance constraint on the state. \Citet{coppens:patrinos:2021} consider a disturbance feedback affine parameterization with conic representable ambiguity sets and demonstrate a tractable reformulation of the DRMPC problem. \citet{mark:liu:2022} consider a similar formulation with a simplified ambiguity set and also establish some performance guarantees for the closed-loop system. \citet{tacskesen:iancu:kocyigit:kuhn:2023} demonstrate that for \emph{unconstrained} linear systems, additive disturbances, and quadratic costs, a linear feedback law is optimal and can be found via a semidefinite program (SDP). \citet{pan:faulwasser:2023} use polynomial chaos expansion to approximate and solve the DRO problem. 

While these new formulations are interesting, there remain important questions about the efficacy of including yet another layer of uncertainty in the MPC problem. For example, what properties should DRMPC provide to the closed-loop system in \cref{eq:fcl}? And what conditions are required to achieve these properties? Due to the rolling horizon nature of DRMPC, distributionally robust closed-loop properties are not necessarily obtained by simply solving a DRO problem. Moreover, if SMPC has an incorrect distribution, the conditions under which DRMPC may provide closed-loop performance benefits relative to this SMPC implementation are not well understood. 

One of the main contributions of this paper is to provide greater insight into these questions. The focus is on the performance benefits and guarantees that may be obtained by considering distributional uncertainty in the cost function \cref{eq:dro_intro}. Chance constraints are therefore not considered in the proposed DRMPC formulation or analysis.

DRMPC is also limited by practical concerns related to the computational cost of solving DRO problems. While these DRMPC problems can often be reformulated as convex optimization problems, in particular SDPs, these optimization problems are often significantly more difficult to solve relative to the quadratic programs (QPs) that are ubiquitous in nominal, robust, and stochastic MPC problems. 

In this work, we consider a DRMPC formulation for linear dynamical models, additive disturbances, robust (convex) constraints, and quadratic costs. This DRMPC formulation uses a Gelbrich ambiguity set with fixed first moment (zero mean) as a conservative approximation for a Wasserstein ball of the same radius \citep{gelbrich:1990}.  
The key contributions of this work are (informally) summarized in the following two categories.
\begin{enumerate}[itemsep = 1mm, topsep = 1mm, leftmargin = 5mm]
\item \textbf{Closed-loop guarantees:} 
    \begin{enumerate}[label=(\theenumi\alph*), itemsep = 1mm, topsep = 1mm, leftmargin = 2mm]
    \item \textbf{Distributionally robust long-term performance.} We establish sufficient conditions for DRMPC, in particular the terminal cost and constraint, such that the closed-loop system satisfies a distributionally robust long-term performance bound (\cref{thm:performance}), i.e., we define a function $C(\dsP)$ such that
        \begin{equation}\label{eq:intro_performance}
            \limsup_{k\rightarrow\infty} \calJ_k(\dsP) \leq \max_{\tilde{\dsP}\in\calP} C(\tilde{\dsP}),
        \end{equation}
        for all $\dsP\in\calP$. This bound is \textit{distributionally robust} because it holds for all distributions $\dsP\in\calP$.  
        
        \item \textbf{Distributionally robust stage-wise performance.}  If the stage cost is also positive definite, we establish that the closed-loop system satisfies a distributionally robust stage-wise performance bound (\cref{thm:resie}), i.e., there exists $\lambda\in(0,1)$ and $c,\gamma>0$ such that
        \begin{equation}\label{eq:intro_resie}
            \ev_{\dsP}\left[\ell\big(\phi(k),\kappa(\phi(k))\big)\right] \leq \lambda^k c|\phi(0)|^2 + \max_{\tilde{\dsP}\in\calP} \gamma C(\tilde{\dsP}),
        \end{equation}
        for all $\dsP\in\calP$.
        Moreover, this result directly implies that the closed-loop system is distributionally robust, \textit{mean-squared} input-to-state stable (ISS) (\cref{cor:mss}), i.e., the left-hand side of \cref{eq:intro_resie} becomes $\ev_{\dsP}\left[|\phi(k)|^2\right]$.
        
        \item \textbf{Pathwise input-to-state stability.} A common approach in MPC design is to select the terminal cost via the discrete-algebraic riccati equation (DARE) for the linear system. Under these conditions, we establish that the closed-loop system is in fact (pathwise) ISS (\cref{thm:iss}), a stronger property than mean-squared ISS.
        
        \item \textbf{Exact long-term performance.}
        Given this stronger property of (pathwise) ISS, we can further establish an \textit{exact} value for the long-term performance of DRMPC based on this terminal cost and the closed-loop disturbance distribution (\cref{thm:performance_opt}), i.e.,
        \begin{equation}\label{eq:intro_exact}
            \lim_{k\rightarrow\infty} \calJ_k(\dsP) = C(\dsP)
        \end{equation}
        for all distributions $\dsP$ supported on $\bbW$. 
        Of particular interest is the fact that this result is \textit{independent} of the choice of ambiguity set $\calP$. Thus, the long-term performance of DRMPC, SMPC, and RMPC are \textit{equivalent} for this choice of terminal cost (\cref{cor:equal}). 
    \end{enumerate}  
    
    \item \textbf{Scalable Newton-type (NT) algorithm.} We present a novel optimization algorithm tailored to solve the DRMPC problem of interest (\cref{alg:NT}). In contrast to Frank-Wolfe (FW) algorithms previously proposed to solve DRO problems (e.g., \citep{nguyen:shafieezadeh:kuhn:mohajerin:2023,sheriff:mohajerin:2023}), the proposed algorithm solves a QP at each iteration instead of an LP. 
    The NT algorithm achieves \textit{superlinear} (potentially quadratic) convergence in numerical experiments (\cref{fig:convergence}) and reduces computation time by more than 50\% compared to solving the DRMPC problem as an LMI optimization problem with state-of-the art solvers, i.e., MOSEK (\cref{fig:comparison}). The proposed NT algorithm and code to generate the figures in this paper can be found in the {\em GitHub Repository}~\url{https://github.com/rdmcallister/drmpc}. 
\end{enumerate}

\paragraph*{\textbf{Organization}} In \cref{s:problem}, we introduce the DRMPC problem formulation and associated DRO problem. In \cref{s:cl}, we present the main technical results on closed-loop guarantees. In \cref{s:cl_proofs}, we provide the technical proofs and supporting lemmata for these results. In \cref{s:algorithms}, we discuss the DRO problem of interest and introduce the proposed NT algorithm. In \cref{s:examples}, we study two examples to demonstrate the closed-loop properties established in \cref{s:cl} and the scalability of the proposed algorithm. 

\paragraph*{\textbf{Notation}} Let $\bbR$ denote the reals and subscripts/superscripts denote the restrictions/dimensions for the reals, i.e., $\bbR^n_{\geq 0}$ is the set of non-negative reals with dimension $n$. The transpose of a matrix $M\in\bbR^{n\times m}$ is denoted $M'$. The trace of a square matrix $M\in\bbR^{n\times n}$ is denoted $\tr(M)$. A positive (semi)definite matrix $M\in\bbR^{n\times n}$ is denoted by $M\succ 0$ ($M\succeq 0$). For $M\succeq 0$, let $|x|_M^2$ denote the quadratic form $x'Mx$. A function $\alpha:\bbR_{\geq 0}\rightarrow\bbR_{\geq 0}$ is said to be in class $\calK$, denoted $\alpha(\cdot)\in\calK$, if $\alpha(\cdot)$ is continuous, strictly increasing, and $\alpha(0)=0$. If $\dsP$ is a probability distribution (measure), $\dsP(A)$ denotes the probability of event $A$. For any (Borel measurable) function $g:\bbW\rightarrow\bbR_{\geq 0}$ of the random variable $w$ with the distribution $\dsP$ and support $\bbW$, expected value is defined as as the (Lebesgue) integral $\ev_{\dsP}\left[g(w)\right] := \int_{\bbW}g(w)\dsP(dw)$. Let $\ev_{\dsP}\left[\cdot\mid x\right]$ denote the expected value given $x$. 
The random variables $w$, $\bfw$, $\bfw_{\infty}$ (and functions thereof) represent specific realizations unless they are found inside an operation such as expected value (i.e., $\ev_{\dsP}\left[\cdot\right]$) or the probability of an event $A$ (i.e., $\dsP(A)$). 

\section{Problem Formulation}\label{s:problem}

We consider the linear system with additive disturbances
\begin{subequations}
    \label{system}
\begin{equation}\label{eq:f}
    x^+ = Ax + Bu + Gw,
\end{equation}
where $x\in\bbR^n$, $u\in\bbU\subseteq\bbR^m$, and $w\in\bbW\subseteq\bbR^q$. We also consider state/input constraints 
\begin{equation}\label{eq:bbZ}
    (x,u) \in \bbZ \subseteq \bbR^n \times \bbU,
\end{equation}
\end{subequations}
and terminal constraint $\bbX_f\subseteq\bbR^n$ that satisfy the following assumption throughout this paper. 
% \begin{assumption}[Convex state-action constraints]\label{as:convex}
%     The sets $\bbU$, $\bbW$, $\bbZ$ and $\bbX_f$ are closed, convex, and contain the origin. The set $\bbU$ is bounded and $\bbW$ contains the origin in its interior. 
% \end{assumption}

\begin{assumption}[Dynamics regularities] \label{as:convex}
The system dynamics~\eqref{eq:f} and state-input sets~\eqref{eq:bbZ} satisfy the following:
\begin{enumerate}[label=(\roman*), itemsep = 1mm, topsep = 1mm, leftmargin = 5mm]

\item {\bf Disturbance statistics:} The disturbances $w$ in~\eqref{eq:f} are zero mean random variables, independent in time, and satisfy $w\in\bbW$ with probability one.

\item {\bf Convex state-input constraints:} The sets $\bbU$, $\bbW$, $\bbZ$, and $\bbX_f$ in \eqref{eq:bbZ} are closed, convex, and contain the origin. The sets $\bbU$ and $\bbW$ are bounded and $\bbW$ contains the origin in its interior. 
\end{enumerate}
\end{assumption}

Since we are permitting input constraints and open-loop unable systems, we require a bounded support $\bbW$ for the disturbance.\footnote{If we consider only open-loop stable systems and chance constraints, we can potentially included unbounded disturbances.} While convex sets are not strictly required for some of the following theoretical results, they are important for efficient computation. 
To ensure constraint satisfaction, we use a disturbance feedback parameterization \citep{goulart:kerrigan:maciejowski:2006} as
\begin{equation}\label{eq:feedbackaffine}
    u(i) = v(i) + \sum_{j=0}^{i-1}M(i,j)w(j),
\end{equation}
in which $v(i)\in\bbR^m$ and $M(i,j)\in\bbR^{m\times q}$.
With this parameterization and a finite horizon $N\geq 1$, the input sequence $\bfu:=(u(0),u(1),\dots,u(N-1))$ is defined as 
\begin{equation}\label{eq:bfu}
    \bfu = \bfM \bfw + \bfv,
\end{equation}
where $\bfw:=(w(0),w(1),\dots,w(N-1))$ is the disturbance trajectory. 
Note that the structure of $\bfM$ must satisfy the following requirements to enforce causality: 
\begin{equation*}
   (\bfM,\bfv)\in \Theta := \left\{(\bfM,\bfv) \ \middle| \ \begin{matrix} \bfM\in\bbR^{Nm\times Nq}, \ \bfv\in\bbR^{Nm}  \\ \ M(i,j)=0 \ \forall j \geq i \end{matrix}\right\}\,.
\end{equation*}
The state trajectory $\bfx:=(x(0),x(1),\dots,x(N))$ is therefore
\begin{equation}\label{eq:bfx}
    \bfx = \bfA x + \bfB\bfv + (\bfB\bfM + \bfG)\bfw,
\end{equation}
and the constraints for this parameterization are given by
\begin{equation*}
	\Pi(x) := \bigcap_{\bfw\in\bbW^N} \left\{
	(\bfM,\bfv)\in\Theta\ \middle| \ \begin{matrix} 
	\textnormal{s.t. } \cref{eq:bfu}, \ \cref{eq:bfx} \\
	(x(k),u(k))\in\bbZ \ \forall \ k \\
	x(N) \in \bbX_f \end{matrix} \right\}\,.
\end{equation*}
That is, if $(\bfM,\bfv)\in\Pi(x)$ then the constraints in \cref{eq:bbZ} are satisfied for all realizations of the disturbance trajectory $\bfw\in\bbW^N$. We also define the feasible set
\begin{equation*}
    \calX := \{x\in\bbR^n \mid \Pi(x)\neq\emptyset \},
\end{equation*}
and note that, by definition, $\Pi(x)$ is nonempty for all $x\in\calX$. To streamline notation, we define 
\begin{equation*}
    \theta := (\bfM,\bfv)\in\Pi(x)\,.
\end{equation*}

\begin{lemma}[Policy constraints]\label{lem:compact}
    If \cref{as:convex}(ii) holds, then $\Pi(x)$ is compact and convex for all $x\in\calX$ and $\calX$ is closed and convex. 
\end{lemma}

See \cref{s:appendix} for the proof. Note that \cref{lem:compact} uses a slightly different formulation and set of assumptions than in \citep{goulart:kerrigan:maciejowski:2006}, and we are therefore able to establish that $\Pi(x)$ is also bounded. Moreover, if $\bbZ$ and $\bbX_f$ are polyhedral and $\bbW$ is a polytope, then $\Pi(x)$ is also a (bounded) polytope, which can be easily expressed via linear constraints \citep[Example 7]{goulart:kerrigan:maciejowski:2006}. We consider quadratic stage and terminal costs defined as
\begin{equation*}
    \ell(x,u) = x'Qx + u'Ru \quad \text{and}\quad  V_f(x) = x'Px,
\end{equation*}
with the following standard assumption. 
\begin{assumption}[Positive semidefinite cost]\label{as:cost}
    The matrices $Q$, $R$, and $P$ are positive semidefinite ($Q,R,P\succeq 0$).
\end{assumption}

For a given input and disturbance trajectory, we have the following deterministic cost function: 
\begin{equation*}
    \Phi(x,\bfu,\bfw) := \sum_{k=0}^{N-1}\ell(x(k),u(k)) + V_f(x(N))\,.
\end{equation*}
If we embed the disturbance feedback parameterization in this cost function, we have from \cref{as:cost}
\begin{align*}
    J(x,\theta,\bfw) & := \Phi(x,\bfM\bfw + \bfv,\bfw) \\
    & = |H_xx+H_u\bfv + (H_u\bfM + H_w)\bfw|^2,
\end{align*}
with constant matrices $H_x$, $H_u$, and $H_w$.
Let $\calM(\bbW)$ denote all probability distributions of $w$ with zero mean and $w\in\bbW$ with probability one, i.e.,
\begin{equation*}
    \calM(\bbW) := \left\{\dsP \ \middle| \ \ev_{\dsP}\left[w\right]=0, \ \dsP\left(w\in\bbW\right)=1\right\}\,.
\end{equation*}
For any distribution $\dsP\in\calM(\bbW^N)$ of $\bfw$ with covariance $\bfSigma:=\ev_{\dsP}\left[\bfw\bfw'\right]$, we have
\begin{equation}
    \label{eq:L_obj}
    L(x,\theta,\bfSigma) := \ev_{\dsP}\left[J(x,\theta,\bfw)\right] = |H_xx+H_u\bfv|^2 
    + \tr\Big((H_u\bfM+H_w)'(H_u\bfM+H_w)\bfSigma\Big)\,. 
\end{equation}
Note that $L(x,\theta,\bfSigma)$ is quadratic in $\theta$ and linear in $\bfSigma$. In SMPC, we minimize $L(x,\theta,\bfSigma)$ for a specific covariance $\bfSigma$ and the current state $x$. 

For DRMPC, we instead consider a worst-case version of the SMPC problem in which $\dsP$ takes the worst value within some ambiguity set. 
To define this ambiguity set, we first consider the Gelbrich ball for the covariance of a single disturbance $w\in\bbR^q$ centered at the nominal covariance $\widehat{\Sigma}\in\bbR^{q\times q}$ with radius $\varepsilon\geq 0$ defined as
\begin{equation*}
    \bbB_{\varepsilon}(\widehat{\Sigma}) := \Big\{\Sigma \succeq 0 ~|~ \ \tr\Big(\widehat{\Sigma} + \Sigma - 2\big(\widehat{\Sigma}^{1/2}\Sigma\widehat{\Sigma}^{1/2}\big)^{1/2}\Big) \leq \varepsilon^2 \Big\}\,.
\end{equation*}
To streamline notation, we define
\begin{equation*}
    \bbB_d := \bbB_{\varepsilon}(\widehat{\Sigma}), \qquad \text{where} \qquad d:=\big(\varepsilon,\widehat{\Sigma}\big)\,.
\end{equation*}
This Gelbrich ball produces the following Gelbrich ambiguity set for the distributions of $w$: 
\begin{equation*}
     \calP_d := \left\{\dsP\in\calM(\bbW) \ \middle| \ \ev_{\dsP}[ww']=\Sigma\in\bbB_d\right\}\,.
\end{equation*}

We further assume that this Gelbrich ambiguity set is compatible with $\bbW$, i.e., all covariances $\Sigma\in\bbB_d$ can be achieved by at least one distribution $\dsP\in\calM(\bbW)$. For example, in the extreme case that $\bbW=\{0\}$ then $\calM(\bbW)$ contains only one distribution with all the weight at zero and the only reasonable Gelbrich ball to consider is $\bbB_d=\{0\}$. Formally, we consider only ambiguity parameters $d\in\calD$ with
\begin{equation*}
    \calD := \left\{ (\varepsilon,\widehat{\Sigma}) \ \middle| \begin{matrix} d=(\varepsilon,\widehat{\Sigma}), \ \varepsilon\geq 0, \ \widehat{\Sigma}\succeq 0, \\
    \ \forall \ \Sigma\in\bbB_{d} \ \exists \ \dsP\in\calM(\bbW) \ \textnormal{s.t. } \ev_{\dsP}[ww']=\Sigma \end{matrix} \right\} \,.
\end{equation*}
Note that $\calD$ depends on $\bbW$, but we suppress this dependence to streamline the notation. If $d\in\calD$, then for any $\Sigma\in\bbB_d$ there exists $\dsP\in\calP_d$ such that $\ev_{\dsP}\left[ww'\right]=\Sigma$.

For the disturbance trajectory $\bfw\in\bbW^N$, we define the following ambiguity set that enforces independence in time:  
\begin{equation*}
     \calP_d^N := \prod_{k=0}^{N-1}\calP_d = \left\{\dsP\in\calM(\bbW^N) \ \middle| \ \begin{matrix}
        \ev_{\dsP}[w(k)w(k)'] \in \bbB_d \\
        w(k) \textrm{ are independent}
    \end{matrix}\right\}\,.
\end{equation*}
We can equivalently represent $\calP_d^N$ as
\begin{equation*}
    \calP_d^N = \left\{\dsP\in\calM(\bbW^N) \ \middle| \ \begin{matrix} \ev_{\dsP}\left[\bfw \bfw'\right]=\bfSigma\in\bbB_d^N \\ w(k) \textrm{ are independent}
    \end{matrix}\right\}
\end{equation*}
in which the set~$\bbB^N_d$ is defined as
\begin{equation}
\label{B^N_d}    
    \bbB^N_d := \Big\{\bfSigma = \textnormal{diag} \big(\big[\Sigma_0 ~\cdots~ \Sigma_{N-1}\big] \big) ~|~ \Sigma_k\in\bbB_{d} \, \forall k\Big\} \,.
    %\left\{\bfSigma = \textnormal{diag}\left( \begin{bmatrix} \Sigma_0 & \cdots & \Sigma_{N-1} \end{bmatrix} \right)    \ \middle| \ \Sigma_k\in\bbB_{d} \ \forall k \right\}\,.
    % \begin{bmatrix}
    % \Sigma_0 & \dots & 0 \\ \vdots & \ddots & \vdots \\
    % 0 & \dots & \Sigma_{N-1}
    % \end{bmatrix} 
\end{equation}

The worst-case expected cost is defined as
\begin{equation}\label{eq:V_d}
    V_d(x,\theta) := \max_{\dsP\in\calP_d^N} \ev_{\dsP}\left[J(x,\theta,\bfw)\right] = \max_{\bfSigma\in\bbB^N_d} L(x,\theta,\bfSigma)\,.
\end{equation}
The two maximization problems in \cref{eq:V_d} are equal because $d\in\calD$. 
We now define the DRO problem for DRMPC as 
\begin{subequations}
\label{eq:dro}
\begin{align}
    V_d^0(x) & := \min_{\theta\in\Pi(x)}\max_{\dsP\in\calP_d^N} \ev_{\dsP}\left[J(x,\theta,\bfw)\right]     \label{eq:dro_E}\\
    & = \min_{\theta\in\Pi(x)} \max_{\bfSigma\in\bbB^N_d} L(x,\theta,\bfSigma) \label{eq:dro_L}\\
    & = \min_{\theta\in\Pi(x)} V_d(x,\theta) = V_d(x,\theta^*), \quad \theta^* \in \theta^0_d(x) \label{eq:dro_theta},
    %, \nonumber
\end{align}    
\end{subequations}
where the function~$L$ in \eqref{eq:dro_L} is defined in \eqref{eq:L_obj}, and the set~$\theta^0_d(x)$ in \eqref{eq:dro_theta} denotes the solution set of the (outer) minimization of the worst-case expected loss~$V_d(x,\theta) := \max_{\bfSigma\in\bbB^N_d} L(x,\theta,\bfSigma)$.  
%and we denote the solution(s) to the outer minimization problem as $\theta^0_d(x)$. 
Note that SMPC ($d=(0,\widehat{\Sigma})$) and RMPC ($d=(0,0)$) are special cases of the optimization problem in \cref{eq:dro}. Thus, all subsequent statements about DRMPC include SMPC and RMPC as special cases of $d\in\calD$.
% The solution(s) to the outer minimization are
% \begin{equation}\label{eq:dro_theta}
%     \theta^0_d(x) := \arg\min_{\theta\in\Pi(x)}\max_{\dsP\in\calP_d^N} \ev_{\dsP}\left[J(x,\theta,\bfw)\right] = \arg\min_{\theta\in\Pi(x)} V_d(x,\theta)
% \end{equation}
Fundamental mathematical properties for this optimization problem are provided in \cref{app:problem}. In \cref{app:problem}, we establish existence (but not uniqueness) of a minimizer for \cref{eq:dro}, continuity of $V^0_d(x)$ w.r.t. $x\in\calX$, and measurability of (set-valued) mapping $\theta^0_d(x)$ w.r.t. $x\in\calX$. 

\section{Closed-loop guarantees: Main results}\label{s:cl}

\subsection{Preliminaries and closed-loop system}

We now define the controller and closed-loop system derived from this DRMPC formulation. We assume state feedback and the control law is defined as the \textit{first} input given by the optimal control law parameterization $\theta^0_d(x)$. Although $\theta^0_d(x)$ may be set-valued, i.e., there are multiple solutions, we assume that some selection rule is applied such that the control law $\kappa_d:\calX\rightarrow\bbU$ is a single-valued function that satisfies
\begin{equation*}
    \kappa_d(x) \in \left\{v^0(0)\mid (\bfM^0,\bfv^0)\in\theta^0_d(x)\right\}\,.
\end{equation*}

With this control law, the closed-loop system is
\begin{equation}\label{eq:cl}
    x^+ = Ax + B\kappa_d(x) + Gw\,.
\end{equation}
Let the function $\phi_d(k;x,\bfw_{\infty})$ denote the closed-loop state of \cref{eq:cl} at time $k\geq 0$, for the initial state $x\in\calX$ and the disturbance trajectory $\bfw_{\infty}\in\bbW^\infty$, i.e., a disturbance trajectory in the $\ell^{\infty}$ space of sequences. Define the infinity norm of the sequence $\bfw_{\infty}$ as $||\bfw_{\infty}||:=\sup_{k\geq 0} |w(k)|$. Note that the deterministic value of $\phi_d(k;x,\bfw_{\infty})$ for a realization of $\bfw_{\infty}\in\bbW^{\infty}$ is a function of $d\in\calD$ via the DRMPC control law. 

The goal of this section is to demonstrate the the closed-loop system in \cref{eq:cl} obtains some desirable properties for the class of distributions considered in $\calP_d$. We consider the set of all distributions for the infinite sequence of disturbances $\bfw_{\infty}$ such that the disturbances are independent in time and their marginal distributions are in $\calP_d$, i.e., we consider the set
\begin{equation*}
    \calP_d^\infty := \prod_{k=0}^{\infty}\calP_d = \left\{\dsP\in\calM(\bbW^\infty) \ \middle| \ \begin{matrix}
        \ev_{\dsP}[w(k)w(k)'] \in \bbB_d \\
        w(k) \textrm{ are independent}
    \end{matrix}\right\} \,.
\end{equation*}

Note that we can also treat all $\dsP\in\calP_d^{\infty}$ as probability measures on the measurable space $(\bbW^{\infty},\mathcal{B}(\bbW^{\infty}))$ in which $\mathcal{B}(\bbW^{\infty})$ denotes the Borel $\sigma$-algebra of $\bbW^{\infty}$. 

An important property for the DRMPC algorithm is robust positive invariance, defined as follows.
\begin{definition}[Robust positive invariance]
A set $X\subseteq\bbR^n$ is robustly positively invariant (RPI) for the system in \cref{eq:cl} if $x^+\in X$ for all $x\in X$, $w\in\bbW$, and $d\in\calD$.  
\end{definition}
\noindent Note that this definition is adapted for DRMPC to consider all possible $d\in\calD$. If we choose $\calD=\{(0,0)\}$ (RMPC) then this definition reduces to the standard definition of RPI found in, e.g., \citep[Def 3.7]{rawlings:mayne:diehl:2020}. Since the control law $\kappa_d:\calX\rightarrow\bbU$ is defined on only the feasible set $\calX$, the first step in the closed-loop analysis is to establish that this feasible set is RPI. 
We define the expected average performance of the closed-loop system for $k\geq 1$, given the initial state $x\in\calX$, ambiguity parameters $d\in\calD$, and the distribution $\dsP\in\calP_d^{\infty}$, as follows. 
\begin{equation*}
    \calJ_k(x,d,\dsP) := \ev_{\dsP}\left[\frac{1}{k}\sum_{i=0}^{k-1}\ell\Big(\phi_d(i;x,\bfw_{\infty}),\kappa_d(\phi_d(i;x,\bfw_{\infty}))\Big)\right]
\end{equation*}

\subsection{Main results and key technical assumptions}

The subsequent analysis uses concepts from nominal, robust, min-max, and stochastic MPC, with a particular focus on results in \citep{goulart:kerrigan:maciejowski:2006,allan:bates:risbeck:rawlings:2017,mcallister:rawlings:2022}.To establish desirable properties for the closed-loop system, we consider the following assumption for the terminal cost $V_f(x)=x'Px$ and constraint $\bbX_f$. This assumption is also used in SMPC and RMPC analysis. 

\begin{assumption}[Terminal cost and constraint]\label{as:term}
    The matrix $P\succeq 0$ is chosen such that there exists $K_f\in\bbR^{m\times n}$ satisfying
    \begin{equation}\label{eq:term}
        P - Q - K_f'RK_f \succeq (A+BK_f)'P(A+BK_f) \,.
    \end{equation}
    Moreover, $\bbX_f$ contains the origin in its interior and $(x,K_fx)\in\bbZ$ and $(A+BK_f)x+Gw\in\bbX_f$ for all $x\in\bbX_f$ and $w\in\bbW$.  
\end{assumption}

Verifying \cref{as:term} is tantamount to finding a stabilizing linear control law $u=K_fx$, i.e., $A+BK_f$ is Schur stable, that satisfies the required constraints $(x,K_fx)\in\bbZ$ within some robustly positive invariant neighborhood of the origin $\bbX_f$. With this stabilizing linear control law, we can then construct an appropriate terminal cost matrix $P$ by, e.g., solving a discrete time Lyapunov equation. 

With this assumption, we can guarantee that the feasible set $\calX$ is RPI and establish the following distributionally robust long-term performance guarantee. This performance guarantee is a distributionally robust version of the stochastic performance guarantee typically derived for SMPC (e.g., \citep{lorenzen:dabbene:tempo:allgower:2016,hewing:wabersich:zeilinger:2020,cannon:kouvaritakis:wu:2009}).  

\begin{theorem}[DR long-term performance]\label{thm:performance}
    If \cref{as:convex,as:cost,as:term} hold, then $\calX$ is RPI for \cref{eq:cl} and 
    \begin{equation}\label{eq:performance}
        \limsup_{k\rightarrow\infty}\calJ_k(x,d,\dsP) \leq \max_{\Sigma\in\bbB_d}\tr(G'PG\Sigma),
    \end{equation}
    for all $x\in\calX$, $d\in\calD$, and $\dsP\in\calP^{\infty}_d$. 
\end{theorem}

\Cref{as:term} is critical to \cref{thm:performance} to ensure that the terminal set is RPI with $u=K_fx$. \Cref{thm:performance} ensures that DRMPC must perform at least as well as this terminal control law in the long run, i.e., the right-hand side of \cref{eq:performance}. Note that SMPC provides a tighter bound than DRMPC if $\widehat{\Sigma}=\Sigma$, but this bound does not necessarily hold for SMPC if $\widehat{\Sigma}\neq\Sigma$. 

\Cref{thm:performance}, however, applies only to the average performance in the limit as $k\rightarrow\infty$. If we are also interested in the transient or stage-wise behavior of the closed-loop system at a given time $k\geq 0$, one can include the following assumption. 
\begin{assumption}[Positive definite stage cost]\label{as:track}
    The matrix $Q$ is positive definite, i.e., $Q\succ 0$. Moreover, the feasible set $\calX$ is bounded or $\bbX_f=\bbR^n$. 
\end{assumption}

By also including \cref{as:track}, we can establish a distributionally robust stage-wise performance guarantee. 

\begin{theorem}[DR stage-wise performance]\label{thm:resie}
     If \cref{as:convex,as:cost,as:term,as:track} hold, then there exist $\lambda\in(0,1)$ and $c,\gamma>0$ such that 
    \begin{equation}\label{eq:resie}
        \ev_{\dsP}\left[\ell(x(k),u(k))\right] \leq \lambda^kc|x|^2 + \gamma\left(\max_{\Sigma\in\bbB_d}\tr(G'PG\Sigma)\right), 
    \end{equation}
    in which $x(k)=\phi_d(k;x,\bfw_{\infty})$, $u(k)=\kappa_d(x(k))$ for all $k\geq 0$, $x\in\calX$, $d\in\calD$, and $\dsP\in\calP^{\infty}_d$.  
\end{theorem}

\Cref{thm:resie} ensures that the effect of the initial condition $x$ on the closed-loop stage cost vanishes exponentially fast as $k\rightarrow\infty$. The persistent term on the right-hand side of \cref{eq:resie} accounts for the continuing effect of the disturbance. We note, however, that this persistent term is a constant that depends on the design of the DRMPC algorithm and does not depend on the actual distribution $\dsP$. Since $Q\succ 0$, we can also establish a the following corollary of \cref{thm:resie}.

\begin{corollary}[DR, mean-squared ISS]\label{cor:mss}
     If \cref{as:convex,as:cost,as:term,as:track} hold, then there exist $\lambda\in(0,1)$ and $c,\gamma>0$ such that 
    \begin{equation}\label{eq:mss}
        \ev_{\dsP}\left[|\phi_d(k;x,\bfw_{\infty})|^2\right] \leq \lambda^kc|x|^2 + \gamma\left(\max_{\Sigma\in\bbB_d}\tr(G'PG\Sigma)\right),
    \end{equation}
    for all $k\geq 0$, $x\in\calX$, $d\in\calD$, and $\dsP\in\calP^{\infty}_d$.  
\end{corollary}
The ISS-style bound in \cref{eq:mss} applies to the \textit{mean-squared} norm of the closed-loop state, a commonly referenced quantity in stochastic stability analysis. Note that \cref{eq:mss} also implies a similar bound for $\ev_{\dsP}\left[|\phi_d(k;x,\bfw_{\infty})|\right]$ via Jensen's inequality.\Cref{thm:performance,thm:resie} and \cref{cor:mss} subsume results for SMPC via $d=(0,\widehat{\Sigma})$. However, the extension of these results from SMPC to DRMPC is nontrivial due to the min-max optimization problem (see \cref{rem:independent}).

In MPC formulations, a common strategy is to choose the terminal cost matrix $P$ according to the DARE, i.e., the cost for the linear-quadratic regulator (LQR) of the unconstrained linear system.\footnote{This strategy minimizes $\tr(G'PG\Sigma)$. See \cref{app:optP}.} Specifically, we now consider the following stronger version of \cref{as:term}. 
 
\begin{assumption}[DARE terminal cost]\label{as:term_lqr}
    The matrices $R$ and $P$ are positive definite ($R,P\succ 0$) and we have 
    \begin{equation}\label{eq:dare}
        P = A'PA - A'PB(R+B'PB)^{-1}B'PA + Q,
    \end{equation}
    and $K_f:=-(R+B'PB)^{-1}B'PA$. Moreover, $(x,K_fx)\in\bbZ$ and $(A+BK_f)x+Gw\in\bbX_f$ for all $x\in\bbX_f$ and $w\in\bbW$. The terminal set $\bbX_f$ contains the origin in its interior. 
\end{assumption}

With this stronger assumption, we can establish significantly stronger properties for the DRMPC controller, similar to results for SMPC reported in \citep[Lemma 4.18]{goulart:2007}. In particular, we can establish that the closed-loop system is (pathwise) ISS. 

\begin{theorem}[Pathwise ISS]\label{thm:iss}
    If \cref{as:convex,as:cost,as:track,as:term_lqr} hold, then for any $d\in\calD$ the origin is (pathwise) ISS for \cref{eq:cl}, i.e., there exist $\lambda\in(0,1)$, $c>0$, and $\gamma(\cdot)\in\calK$ such that
    \begin{equation}\label{eq:iss}
        |\phi_d(k;x,\bfw_{\infty})| \leq \lambda^kc|x| + \gamma(||\bfw_{\infty}||),
    \end{equation}
    for all $k\geq 0$, $x\in\calX$, and $\bfw_{\infty}\in\bbW^{\infty}$. 
\end{theorem}

The property of (pathwise) ISS in \cref{thm:iss} is notably stronger than mean-squared ISS in \cref{cor:mss}. The key distinction is that the persistent term on the right-hand side of \cref{eq:iss} is specific to a given realization of the disturbances trajectory $\bfw_{\infty}$, while the persistent term in \cref{cor:mss} depends only on the DRMPC design. If $\bfw_{\infty}=\mathbf{0}$ then \cref{eq:iss} implies that the origin is exponentially stable. By contrast, the weaker restriction on the terminal cost in \cref{as:term} does \emph{not} ensure that the closed-loop system is ISS. We demonstrate this fact in \cref{s:examples} via a counterexample. 

We now consider a class of disturbances that are both independent \emph{and} identically distributed (i.i.d.) in time. We also require that arbitrarily small values of these disturbances occur with nonzero probability. Specifically, we define the following set of distributions. 
\begin{equation*}
    \calQ := \prod_{k=0}^{\infty}\left\{\dsP\in\calM(\bbW) \ \middle| \ \forall \ \varepsilon>0, \ \dsP(|w|\leq \varepsilon) >0\right\}
\end{equation*}
Note that $\calQ$ includes most distributions of interest such as uniform, truncated Gaussian, and even finite distributions with $\dsP(w=0)>0$. For this class of disturbances, we have the following exact long-term performance guarantee.

\begin{theorem}[Exact long-term performance]\label{thm:performance_opt}
    If \cref{as:convex,as:cost,as:track,as:term_lqr} hold, then% we have 
    \begin{equation}\label{eq:performance_optK}
        \lim_{k\rightarrow\infty}\calJ_k(x,d,\dsP) = \tr(G'PG\Sigma) \textrm{ with } \Sigma=\ev_{\dsP}\left[w(i)w(i)'\right],
    \end{equation}
    for all $x\in\calX$, $d\in\calD$, and $\dsP\in\calQ$.
\end{theorem}

Note that \cref{eq:performance_optK} provides an exact value for the long-term performance based on the distribution of the disturbance in the closed-loop system. By contrast, \cref{eq:performance} provides a conservative and constant bound based on the design parameter $d\in\calD$. Furthermore, the values of $d\in\calD$ do \emph{not} affect the long-term performance in \cref{eq:performance_optK}. By recalling that SMPC and RMPC are special cases of DRMPC, we have the following corollary of \cref{thm:performance_opt}.

\begin{corollary}[DRMPC versus SMPC]\label{cor:equal}
    If \cref{as:convex,as:cost,as:track,as:term_lqr} hold, then the long-term performance of DRMPC, SMPC ($\varepsilon=0$), and RMPC ($\varepsilon=0$, $\widehat{\Sigma}=0$) are equivalent, i.e., 
    \begin{equation*}
        \lim_{k\rightarrow\infty}\calJ_k\Big(x,(\varepsilon,\widehat{\Sigma}),\dsP\Big) = \lim_{k\rightarrow\infty}\calJ_k\Big(x,(0,\widehat{\Sigma}),\dsP\Big) = \lim_{k\rightarrow\infty}\calJ_k\Big(x,(0,0),\dsP\Big),
    \end{equation*}
    for all $x\in\calX$, $(\varepsilon,\widehat{\Sigma})\in\calD$, and $\dsP\in\calQ$.
\end{corollary}

\Cref{cor:equal} provides notable insight into the potential benefits of DRMPC relative to SMPC if SMPC has an incorrect distribution. Specifically, the additional requirements of \cref{as:term_lqr} ensure that DRMPC does \emph{not} provide any long-term performance benefit relative to SMPC (or RMPC) regardless of the distrubance's true distribution $\dsP$. Thus, DRMPC can provide a long-term performance benefit only in the case that \cref{as:term_lqr} does \emph{not} hold and the covariance $\widehat{\Sigma}$ used in SMPC is sufficiently inaccurate. 

Although selecting $P$ to satisfy \cref{eq:dare} is a standard design method in MPC, there are also systems in which one cannot satisfy the requirements of \cref{as:term_lqr} for a given $Q,R\succ 0$. In particular, if the origin is sufficiently close to (or on) the boundary of $\bbZ$, then satisfying all of the requirements in \cref{as:term_lqr} is typically not possible. In chemical process control, for example, processes often operate near input constraints (e.g., maximum flow rates) to ensure high throughput for the process. Thus, the terminal cost and constraint are chosen to satisfy only the weaker condition in \cref{as:term}. In this case, there is a possibility that DRMPC produces superior long-term performance relative to SMPC if the covariance $\widehat{\Sigma}$ used in SMPC is sufficiently different from the true covariance of the disturbance, i.e., $\widehat{\Sigma}\neq\Sigma$. We therefore focus on examples in \cref{s:examples} that satisfy \cref{as:term}, but cannot satisfy \cref{as:term_lqr}. 

\begin{remark}[Detectable stage cost]\label{rem:detectable}
    We can also weaken \cref{as:track} to $Q\succeq 0$ and $(A,Q^{1/2})$ detectable. By defining an input-output-to-state stability (IOSS) Lyapunov function we can apply the same approach used for nominal MPC in, e.g., \citep[Thm. 2.24]{rawlings:mayne:diehl:2020}, to establish \cref{thm:resie,cor:mss,thm:iss,thm:performance_opt,cor:equal} for DRMPC under this weaker restriction for $Q$.
\end{remark}

\section{Closed-loop guarantees: Technical proofs}\label{s:cl_proofs} 

\subsection{Distributionally robust long-term performance}\label{ss:standard}

To establish \cref{thm:performance}, we begin by establishing that feasible set $\calX$ is RPI and providing a distributionally robust expected cost decrease condition. 

\begin{lemma}[DR cost decrease]\label{lem:costdec}
    If \cref{as:convex,as:cost,as:term} hold, then the feasible set $\calX$ is RPI for \cref{eq:cl} and 
    \begin{equation}\label{eq:cost_dec}
        \ev_{\dsP}\left[V^0_d(x^+)\right] \leq V^0_d(x) - \ell(x,\kappa_d(x)) + \max_{\Sigma\in\bbB_d}\tr(G'PG\Sigma)
    \end{equation}
    for all $\dsP\in\calP_d$, $d\in\calD$, and $x\in\calX$. 
\end{lemma}

\begin{proof}
    Choose $x(0)\in\calX$ and $d\in\calD$. Define $(\bfM^0,\bfv^0)=\theta^0\in\theta^0_d(x(0))$. 
    Consider the subsequent state $x(1)=Ax(0)+Bv^0(0)+Gw(0)$ for some $w(0)\in\bbW$. For the state $x(1)$, we choose a candidate solution 
    \begin{equation}\label{eq:candidate}
        \tilde{\theta}^+(w(0)) = \Big(\tilde{\bfM}^+,\tilde{\bfv}^+(w(0))\Big)
    \end{equation}
    such that the open-loop input trajectory remains the same as the previous optimal solution, i.e., $u(1)=v^0(1) + M^0(1,0)w(0)=\tilde{v}^+(1)$ and
    \begin{align}
        u(k) & = v^0(k) + \sum_{j=0}^{k-1}M^0(k,j)w(j) \nonumber \\
        & = \tilde{v}^+(k) + \sum_{j=1}^{k-1}\tilde{M}^+(k,j)w(j), \label{eq:u_cand}
    \end{align}
    for all $k\in\{2,\dots,N-1\}$ and $\bfw\in\bbW^N$. With this choice of parameters, the open-loop state trajectories $x(k)$ are also the same for all $k\in\{1,\dots,N-1\}$ and $\bfw\in\bbW^N$. 
    The candidate solution is therefore
    \begin{equation*}
        \begin{split}
            & \tilde{\bfM}^+ =
            \begingroup % keep the change local
            \setlength\arraycolsep{1pt}
            \begin{bmatrix}
        	0 & \cdots & \cdots & 0 & 0 \\
        	M^0(2,1) & 0 & \cdots & 0 & 0\\
        	\vdots & \ddots & \ddots & \vdots & \vdots \\
        	M^0(N-1,1) & \cdots & M^0(N-1,N-2) & 0 & 0 \\
            \tilde{M}^+(N,1) & \cdots & \tilde{M}^+(N,N-2) & \tilde{M}^+(N,N-1) & 0
    	\end{bmatrix}
        \endgroup
        \end{split}
    \end{equation*}
    \begin{equation*}
            \tilde{\bfv}^+(w(0)) = \begin{bmatrix}
            v^0(1) + M^0(1,0)w(0) \\ v^0(2) + M^0(2,0)w(0) \\ \vdots \\ v^0(N-1) + M^0(N-1,0)w(0) \\ \tilde{v}^+(N)
        \end{bmatrix}
    \end{equation*}
    in which that last rows of $\tilde{\bfM}^+$ and $\tilde{\bfv}^+(w(0))$ are not yet defined. We define these last rows by the terminal control law $u(N)=K_fx(N)$. Specifically, we have
    \begin{equation}\label{eq:warmstart_N}
        K_fx(N) = \tilde{v}^+(N) + \sum_{i=1}^{N-1} \tilde{M}^+(N,i)w(i) \,.
    \end{equation}
    By definition of $x(N)$, we have
    \begin{equation*}
        x(N) = A^{N-1}x(1) + \sum_{i=1}^{N-1}A^{N-1-i}\Big(Bu(i) + Gw(i)\Big) \,.
    \end{equation*}
   We then substitute the values of $u(i)$ for the candidate solution in \cref{eq:u_cand} to give
    \begin{multline*}
        x(N) = A^{N-1}x(1) + \sum_{i=1}^{N-1}A^{N-1-i}B\Big(v^0(i) + M^0(i,0)w(0)\Big) \\ + \sum_{i=1}^{N-1}A^{N-1-i}\left(\sum_{j=1}^{i-1}BM^0(i,j)w(j) + G w(i)\right) \,.
    \end{multline*}
    With some manipulation, we can therefore define
    \begin{align*}
        \tilde{v}^+(N) & = K_fA^{N-1}x(1)
        + \sum_{i=1}^{N-1}K_fA^{N-1-i}B\Big(v(i) + M^0(i,0)w(0)\Big) \\
        \tilde{M}^+(N,i) & = K_f\bigg(\sum_{j=i+1}^{N-1}A^{N-1-j}BM^0(j,i) + A^{N-1-i}G\bigg),
    \end{align*}
    to satisfy \cref{eq:warmstart_N}. Note that $\tilde{\bfM}^+$ is independent of $w(0)$ and $\tilde{\bfv}^+(w(0))$ is an affine function of $w(0)$. 
    
    We first establish that this candidate solution is feasible for any $w(0)\in\bbW$ and that $\calX$ is RPI. Since $(\bfM^0,\bfv^0)\in\Pi(x(0))$, then $(x(k),u(k))\in\bbZ$ for all $k\in\{1,\dots,N-1\}$ and $x(N)\in\bbX_f$ for all $\bfw\in\bbW^N$. From \cref{as:term}, we also have that $(x(N),K_fx(N))\in\bbZ$ for all $\bfw\in\bbW^N$. Therefore, $(x(N),u(N))\in\bbZ$ and $x(N+1)=(A+BK_f)x(N) + Gw(N)\in\bbX_f$ for all $w(N)\in\bbW$ by \cref{as:term}. Thus, $(\tilde{\bfM}^+,\tilde{\bfv}^+)\in\Pi(x(1))$. Since $\Pi(x(1))\neq\emptyset$, we also know that $x(1)\in\calX$ for any $w(0)\in\bbW$. Since the choice of $x(0)\in\calX$ and $d\in\calD$ was arbitrary, we have that $\calX$ is RPI. 
    
    Choose $\bfw=(w(0),\dots,w(N-1))\in\bbW^N$ and define $\bfw^+=(w(1),\dots,w(N))$ with some additional $w(N)\in\bbW$. We have that
    \begin{multline}\label{eq:J_diff}
        J(x(1),\tilde{\theta}^+(w(0)),\bfw^+) - J(x(0),\theta^0,\bfw) = -\ell(x(0),v^0(0)) \\ + V_f(x(N+1)) - V_f(x(N)) + \ell(x(N),K_fx(N)) \,.
    \end{multline}
    We define
    \begin{align}\label{eq:Sigma_plus}
        \bfSigma^+ & = \arg\max_{\bfSigma\in\bbB_d^N} L(x(1),\tilde{\theta}^+(w(0)),\bfSigma) \\ & = \arg\max_{\bfSigma\in\bbB_d^N} \tr\left((H_u\tilde{\bfM}^+ + H_w)'(H_u\tilde{\bfM}^+ + H_w)\bfSigma\right),\nonumber
    \end{align}
    in which equality holds because $|H_xx+H_u\bfv|^2$ is a constant with respect to $\bfSigma$ and therefore does not affect the worst-case value of $\bfSigma$. Note that $\bfSigma^+$ is independent of $w(0)$ because $\tilde{\bfM}^+$ is independent of $w(0)$, which is crucial to the subsequent analysis.
    We also define the distribution $\dsQ\in\calM(\bbW^N)$ for $\bfw^+$ such that $\ev_{\dsQ}[(\bfw^+)(\bfw^+)']=\bfSigma^+$. Note that such a $\dsQ$ exists because $d\in\calD$. For this distribution, we take the expected value of each side of \cref{eq:J_diff} to give
    \begin{multline*}
         \ev_{\dsQ}\left[J(x(1),\tilde{\theta}^+(w(0)),\bfw^+) - J(x(0),\theta^0,\bfw)\mid w(0)\right] = \\ \ev_{\dsQ}\left[V_f(x(N+1)) - V_f(x(N)) + \ell(x(N),K_fx(N))\mid w(0)\right]
         -\ell(x(0),v^0(0)) \,. 
    \end{multline*}
    From \cref{as:term} and the fact that $x(N)\in\bbX_f$ for all $\bfw\in\bbW^N$, we have that
    \begin{equation*}
        \ev_{\dsQ}\left[V_f(x(N+1)) - V_f(x(N)) + \ell(x(N),K_fx(N))\mid w(0)\right] \leq \tr(G'PG\Sigma_N) \leq \delta_d
    \end{equation*}
    in which $\Sigma_N=\ev_{\dsQ}[w(N)w(N)']\in\bbB_d$ and $\delta_d:=\max_{\Sigma\in\bbB_d}\tr(G'PG\Sigma)$.  
    From the definition of $\bfSigma^+$ and optimality, we have
    \begin{equation*}
        V_d\Big(x(1),\tilde{\theta}^+(w(0))\Big)  = \ev_{\dsQ}\left[J(x(1),\tilde{\theta}^+(w(0)),\bfw^+)\mid w(0)\right],
    \end{equation*} 
    and therefore
    \begin{equation}\label{eq:E_Q}
        V_d\Big(x(1),\tilde{\theta}^+(w(0))\Big) \leq \ev_{\dsQ}\left[J(x(0),\theta^0,\bfw)\mid w(0)\right] - \ell(x(0),v^0(0)) + \delta_d \,.
    \end{equation}

    Choose any $\dsP\in\calP_d$ for the distribution of $w(0)$.
    From the definition of $\dsP$ and $\dsQ$, we have
    \begin{equation*}
        \ev_{\dsP}\left[\ev_{\dsQ}\left[J(x(0),\theta^0,\bfw)\mid w(0)\right]\right] \leq V_d(x(0),\theta^0) = V^0_d(x(0))
    \end{equation*}
    because $\theta\in\theta^0_d(x(0))$. We take the expected value of \cref{eq:E_Q} with respect to $\dsP$ and use this inequality to give
    \begin{equation}\label{eq:E_QP}
        \ev_{\dsP}\left[V_d\Big(x(1),\tilde{\theta}^+(w(0))\Big)\right] \leq  V^0_d(x(0)) - \ell(x(0),v^0(0)) + \delta_d \,.
    \end{equation}
    By optimality, we have
    \begin{equation*}
        V_d^0(x(1)) \leq V_d\Big(x(1),\tilde{\theta}^+(w(0))\Big) \,.
    \end{equation*}
    We combine this inequality with \cref{eq:E_QP} and substitute in $x=x(0)$, $x^+=x(1)$, $\kappa_d(x)=v^0(0)$ to give \cref{eq:cost_dec}. Note that the choices of $\dsP\in\calP_d$, $d\in\calD$, and $x(0)\in\calX$ were arbitrary and therefore \cref{eq:cost_dec} holds for all values in these sets. 
\end{proof}

The main difference between \cref{lem:costdec} and the typical expected cost decrease condition for SMPC is that \cref{eq:cost_dec} holds for all $\dsP\in\calP_d$, i.e., the inequality is distributionally robust. 

\begin{remark}[Key distinction from SMPC analysis]\label{rem:independent}
    In the proof of \cref{lem:costdec}, we show that for the candidate solution the worst case covariance, defined as $\bfSigma^+$ in \cref{eq:Sigma_plus}, is \textit{independent} of $w(0)$. This independence is essential to establish \cref{lem:costdec} and is derived from the affine control law parameterization. As such, these results may not hold for other (nonlinear) control law parameterizations. 
\end{remark}

We can then apply \cref{lem:costdec} to prove \cref{thm:performance}. 

\begin{proof}[Proof of \cref{thm:performance}]
    Choose $x\in\calX$, $d\in\calD$, and $\dsP\in\calP_d^{\infty}$. Define $x(i)=\phi_d(i;x,\bfw_{\infty})$ and $u(i)=\kappa_d(x(i))$. From \cref{lem:costdec}, we have that $\calX$ is RPI and
    \begin{equation}\label{eq:costdec_Q}
        \ev_{\dsQ}\left[V^0_d(x(i+1))\mid x(i)\right] \leq V^0_d(x(i)) - \ell(x(i),u(i)) + \max_{\Sigma\in\bbB_d}\tr(G'PG\Sigma)
    \end{equation}
    for all $\dsQ\in\calP_d$. Let $\delta_d:=\max_{\Sigma\in\bbB_d}\tr(G'PG\Sigma)$ to streamline notation. From the law of total expectation and \cref{eq:costdec_Q}, we have
    \begin{equation*}
    \ev_{\dsP}\left[\ell(x(i),u(i))\right] \leq \ev_{\dsP}\left[V^0_d(x(i))\right] - \ev_{\dsP}\left[V^0_d(x(i+1))\right] + \delta_d \,.
    \end{equation*}
    We sum both sides of this inequality from $i=0$ to $i=k-1$ and divide by $k\geq 1$ to give
    \begin{equation*}
        \calJ_k(x,d,\dsP) \leq \frac{\ev_{\dsP}\left[V^0_d(x(0))\right] - \ev_{\dsP}\left[V^0_d(x(k))\right]}{k} + \delta_d \,.
    \end{equation*}
    Note that $V^0_d(x(k))\geq 0$. We take the $\limsup$ as $k\rightarrow \infty$ of both sides of the inequality to give \cref{eq:performance}. 
\end{proof}

\subsection{Distributionally robust stage-wise performance}

To establish \cref{thm:resie}, we first establish the following upper bound for the optimal cost function.  
\begin{lemma}[Upper bound]\label{lem:upperbound}
    If \cref{as:convex,as:cost,as:term,as:track} hold, then there exists $c_2>0$ such that
    \begin{equation}\label{eq:upper_bound}
        V^0_d(x) \leq c_2|x|^2 + N\left(\max_{\Sigma\in\bbB_d}\tr(G'PG\Sigma)\right),
    \end{equation}
    for all $d\in\calD$ and $x\in\calX$.
\end{lemma}

\begin{proof}
    For any $x\in\bbX_f$, we define the control law as $u=K_fx$ from \cref{as:term}. Therefore,
    \begin{equation*}
        \bfu = \bfK_f\bfx, \quad \bfx = (I-\bfB\bfK_f)^{-1}(\bfA x + \bfG \bfw)  
    \end{equation*}
    in which $\bfK_f := \begin{bmatrix} I_N\otimes K_f & 0\end{bmatrix}$. 
    Note that the inverse $(I-\bfB\bfK_f)^{-1}$ exists because $\bfB\bfK_f$ is nilpotent (lower triangular with zeros along the diagonal). We represent this control law as $\theta_f(x)=(\bfM_f,\bfv_f(x))$ so~that
    \begin{equation}\label{eq:theta_f}
        \begin{split}
        \bfv_f(x) & = \bfK_f(I-\bfB\bfK_f)^{-1}\bfA x \\ \bfM_f & = \bfK_f(I-\bfB\bfK_f)^{-1}\bfG 
        \end{split}
    \end{equation}
    We have from \cref{as:term} that this control law ensures that $(x(k),u(k))\in\bbZ$ and $x(k+1)\in\bbX_f$ for all $k\in\{0,\dots,N-1\}$. Therefore $\theta_f(x)\in\Pi(x)$ for all $x\in\bbX_f$ and $d\in\calD$. Choose any $x\in\bbX_f$ and $d\in\calD$. Choose any $\bfSigma\in\bbB^N_d$ and corresponding $\dsP\in\calP_d^N$ such that $\ev_{\dsP}\left[\bfw\bfw'\right]=\bfSigma$. From \cref{as:term}, we have that
    \begin{equation*}
        \ev_{\dsP}\left[V_f(x(k+1)) - V_f(x(k)) + \ell(x(k),K_fx(k))\right] \leq \tr(G'PG\Sigma_k) \leq \delta_d
    \end{equation*}
    in which $x(k+1)=(A+BK_f)x(k)+Gw(k)$ for all $k\in\{0,1,\dots,N-1\}$, $x(0)=x$, and $\delta_d:=\max_{\Sigma\in\bbB_d}\tr(G'PG\Sigma)$. We sum both sides of this inequality from $k=0$ to $k=N-1$ and rearrange to give
    \begin{equation*}
        L(x,\theta_f(x),\bfSigma) \leq V_f(x(0)) + N\delta_d
    \end{equation*}
    for all $\bfSigma\in\bbB^N_d$. 
    Therefore,
    \begin{equation*}
        V^0_d(x) \leq V_d(x,\theta_f(x)) \leq V_f(x) + N\delta_d \leq \bar{\lambda}_P|x|^2 + N\delta_d
    \end{equation*}
    in which $\bar{\lambda}_P$ is the maximum eigenvalue of $P$ for all $x\in\bbX_f$. If $\bbX_f=\bbR^n$, the proof is complete because $\calX\subseteq\bbR^n$. Otherwise, we use the fact that $\calX$ is bounded to extend this bound to all $x\in\calX$. Define the function
    \begin{equation*}
        F(x) = \sup \left\{V^0_d(x) - N\delta_d \ \middle| \ d\in\calD \right\} \,.
    \end{equation*}
    Since $\bbW$ is bounded, $\calD$ is bounded as well (\cref{lem:calD}). 
    Therefore, $F(x)$ is finite for all $x\in\calX$. We further define
    \begin{equation*}
        r := \sup\left\{F(x)/(|x|^2) \ \middle| \ x\in\calX\setminus \bbX_f\right\}
    \end{equation*}
    Note that since $\bbX_f$ contains the origin in its interior and $F(x)$ is finite for all $x\in\calX$, $r$ exists and is finite. Therefore,
    \begin{equation*}
        V^0_d(x) \leq F(x) + N\delta_d \leq r|x|^2 + N\delta_d
    \end{equation*}
    for all $x\in\calX\setminus\bbX_f$. We define $c_2:=\max\{r,\bar{\lambda}_P\}$ and substitute in the definition of $\delta_d$ to complete the proof. 
\end{proof}

With this upper bound, we prove \cref{thm:resie} by using $V_d^0(x)$ as a Lyapunov-like function.  

\begin{proof}[Proof of \cref{thm:resie}]
    Since $Q\succ 0$, there exists $c_1>0$ such that for all $x\in\calX$:
    \begin{equation*}
        c_1|x|^2 \leq \ell(x,\kappa_d(x)) \leq V^0_d(x) \,.
    \end{equation*}
    Let $\delta_d=\max_{\Sigma\in\bbB_d}\tr(G'PG\Sigma)$. From \cref{eq:upper_bound}, we have
    \begin{equation*}
        -|x|^2 \leq -(1/c_2)V^0_d(x) + (N\delta_d/c_2)
    \end{equation*}
    We combine this inequality and the lower bound for $\ell(x,\kappa_d(x))$ with \cref{eq:cost_dec} to give
    \begin{equation}\label{eq:exp_costdec}
        \ev_{\dsQ}\left[V^0_d(x^+)\mid x\right] \leq \lambda V^0_d(x) + (1+Nc_1/c_2)\delta_d,
    \end{equation}
    in which $\lambda = (1-c_1/c_2)\in(0,1)$ and $x^+=Ax+B\kappa_d(x)+Gw$ for all $\dsQ\in\calP_d$, $d\in\calD$, and $x\in\calX$.  

    Choose $x\in\calX$, $d\in\calD$, and $\dsP\in\calP_d^{\infty}$. Define $x(k)=\phi_d(k;x,\bfw_{\infty})$ and $u(k)=\kappa_d(x(k))$. Since $\calX$ is RPI for the closed-loop system (\cref{lem:costdec}), $x(k)\in\calX$ for all $x\in\calX$, $d\in\calD$, $\bfw_{\infty}\in\bbW^{\infty}$, and $k\geq 0$. From \cref{eq:exp_costdec} we have
    \begin{equation*}%\label{eq:exp_costdec_k}
        \ev_{\dsQ}\left[V^0_d(x(k+1))\mid x(k)\right] \leq \lambda V^0_d(x(k)) + (1+Nc_1/c_2)\delta_d
    \end{equation*}
    for all $\dsQ\in\calP_d$. 
    We take the expected value of this inequality with respect to $\dsP$ and the corresponding $\dsQ$ to give
    \begin{equation*}%\label{eq:E_exp_costdec}
         \ev_{\dsP}\left[V^0_d(x(k+1))\right] \leq \lambda \ev_{\dsP}\left[V^0_d(x(k))\right] + (1+Nc_1/c_2)\delta_d \,.
    \end{equation*}
    By iterating, we have
    \begin{equation*}\ev_{\dsP}\left[V^0_d(x(k))\right] \leq \lambda^k V^0_d(x(0)) + \frac{1+Nc_1/c_2}{1-\lambda} \delta_d \,.
    \end{equation*}
    Use the lower/upper bounds for $V^0_d(\cdot)$, 
    rearrange, and define $c:=c_2/c_1$, $\gamma := N + (c_1^{-1}+Nc_2^{-1})/(1-\lambda)$ to give \cref{eq:resie}. 
\end{proof}

\subsection{Pathwise input-to-state stability}

To establish \cref{thm:iss}, we first establish the following interesting property for the DRMPC control law within the terminal region $\bbX_f$, similar to \citep[Lemma 4.18]{goulart:2007}.
\begin{lemma}[Terminal control law]\label{lem:K_f}
    If \cref{as:convex,as:cost,as:track,as:term_lqr} hold, then 
    \begin{equation*}
        \kappa_d(x):= K_fx,
    \end{equation*}
    for all $x\in\bbX_f$ and $d\in\calD$. Moreover, $\bbX_f$ is RPI for \cref{eq:cl}. 
\end{lemma}

\begin{proof}
    From the definition of $P$ and $K_f$ in \cref{as:term_lqr} and any $\dsP\in\calM(\bbW^{N})$, we have
    \begin{equation*}
        \ev_{\dsP}\left[ \Phi(x,\bfu,\bfw) \right] = |x|_P^2 + \ev_{\dsP}\left[|\bfu-\bfK_f\bfx|^2_{\bfS}\right] + \ev_{\dsP}\left[ |\bfw|^2_{\bfP}\right], %\label{eq:EPhi}
    \end{equation*}
    in which $\bfK_f := \begin{bmatrix}I_N\otimes K_f & 0 \end{bmatrix}$, $\bfS := I_N \otimes (R+B'PB)$ $\bfP := I_N\otimes (G'PG)$ (see \citep[eq. (4.46)]{goulart:2007}).
    Using the control law parameterization $\theta=(\bfM,\bfv)$, we have
    \begin{multline}\label{eq:EJ}
        \ev_{\dsP}\left[ J(x,\theta,\bfw) \right] = |x|_P^2 + |\bfv - \bfK_f\bfA x - \bfK_f\bfB\bfv|^2_{\bfS} \\
        +  \ev_{\dsP}\left[|(\bfM-\bfK_f\bfB\bfM - \bfK_f\bfG)\bfw|^2_{\bfS}\right]  + \ev_{\dsP}\left[ |\bfw|^2_{\bfP}\right] \,.
    \end{multline}
    
    We have that the optimal solution is bounded by
    \begin{equation}
        \min_{\theta\in\Pi(x)}\max_{\dsP\in\calP_d^N} \ev_{\dsP}\left[ J(x,\theta,\bfw) \right] \geq |x|_P^2 + \max_{\bfSigma\in\bbB_{d}^N}\tr\Big(\bfP\bfSigma\Big), \label{eq:Vopt_lb}
    \end{equation}
    for all $x\in\calX$. This lower bound is obtained by
    \begin{equation*}
        \bfv_c(x) := (I - \bfK_f\bfB)^{-1}\bfK_f\bfA x \qquad \bfM_c := (I-\bfK_f\bfB)^{-1}\bfK_f\bfG
    \end{equation*}
    and $\theta_c(x)=(\bfM_c,\bfv_c(x))$ for any $\dsP\in\calP_d^N$. Note that the inverse $\left(I-\bfK_f\bfB\right)^{-1}$ exists because $\bfK_f\bfB$ is nilpotent (lower triangular with zeros along the diagonal). By application of the matrix inversion lemma, we have that $\theta_c(x)=\theta_f(x)$ in \cref{eq:theta_f}. Therefore $\theta_c(x)=\theta_f(x)\in\Pi(x)$ for all $x\in\bbX_f$.      
    Moreover, the solution $\bfv_c(x)$ is unique because $\bfS\succ 0$. Therefore,
    \begin{equation*}
        \kappa_d(x) = v^0(0;x) = v_c(0;x) = K_fx
    \end{equation*}
    is the unique control law for all $x\in\bbX_f$. 
    Since $\kappa_d(x)=K_fx$ for all $x\in\bbX_f$ and $\bbX_f$ is RPI for the system $x^+ = (A+BK_f)x + Gw$, we have that $\bbX_f$ is also RPI for \cref{eq:cl}. 
\end{proof}

\Cref{lem:K_f} ensures that within the terminal region, the DRMPC control law is \textit{equivalent} to the LQR control law defined in \cref{as:term_lqr}. Moreover, this controller is the same regardless of the choice of $d\in\calD$ and renders the terminal set RPI. Therefore, once the state of the system reaches $\bbX_f$, the control laws for DRMPC, SMPC, RMPC, and LQR are the same. We can now prove \cref{thm:iss}. 

\begin{proof}[Proof of \cref{thm:iss}]
    Choose $d\in\calD$ and $x\in\calX$. 
    We define $(\bfM^0,\bfv^0)=\theta^0\in\theta^0_d(x)$ and the corresponding candidate solution $(\tilde{\bfv}^+(w),\tilde{\bfM}^+) = \tilde{\theta}^+(w)$ defined in \cref{eq:candidate}. 
    Recall that $\tilde{\bfv}^+(w)$ is an affine function of $w$, i.e., 
    \begin{equation*}
        \tilde{\bfv}^+(w) = c + Zw
    \end{equation*}
    in which $c\in\bbR^{Nm}$ and $Z\in\bbR^{Nm\times q}$ are fixed quantities for a given $\theta^0$.  
    Let 
    \begin{equation*}
        \widehat{x}^+=Ax+B\kappa_d(x) = Ax+Bv^0(0) \,.
    \end{equation*}
    From the Proof of \cref{lem:costdec}, we have that
    \begin{equation}
        \ev_{\dsQ}\left[V_d(\widehat{x}^++Gw,\tilde{\theta}^+(w))\right] \leq V_d^0(x) - \ell(x,v(0)) + \delta_d, \label{eq:costdec_xhat}
    \end{equation}
    for all $\dsQ\in\calP_d$ in which $\delta_d:=\max_{\Sigma\in\bbB_d}\tr(G'PG\Sigma)$.    
    Define
    \begin{align*}
        \bfSigma^+ & := \arg\max_{\bfSigma\in\bbB_{d}^N} L(\widehat{x}^++Gw,\tilde{\theta}^+(w),\bfSigma) \\
        & = \arg\max_{\bfSigma\in\bbB_{d}^N} \tr\left((H_u\tilde{\bfM}^++H_w)'(H_u\tilde{\bfM}^++H_w)\bfSigma\right),
    \end{align*}
    where the equality holds because $|H_xx+H_u\bfv|^2$ is a constant with respect to $\bfSigma$. Note that $\bfSigma^+$ is independent of $w$ because $\tilde{\bfM}^+$ is independent of $w$. We also define $\dsP\in\calP_d^N$ such that $\bfSigma^+=\ev_{\dsP}\left[(\bfw^+)(\bfw^+)'\right]$. 
    By applying \cref{eq:EJ}, we have
    \begin{align}
         & L(\widehat{x}^+,\tilde{\theta}^+(0),\tilde{\bfSigma}^+) - L(\widehat{x}^++Gw,\tilde{\theta}^+(w),\tilde{\bfSigma}^+) \nonumber \\
         & = |\widehat{x}^+|_P^2 + |c-\bfK_f\bfA\widehat{x}^+ - \bfK_f\bfB c|_{\bfS}^2 - |\widehat{x}^++Gw|_P^2 \nonumber \\
         & \quad - |c+Zw - \bfK_f\bfA (\widehat{x}^++Gw) - \bfK_f\bfB(c+Zw)|^2_{\bfS},\label{eq:L_diff}
    \end{align}
    and note that the terms involving $\dsP$ and $\tilde{\bfM}^+$ in \cref{eq:EJ} do not change with $w$ and therefore vanish in this difference. By the definition of $\bfSigma^+$ and optimality, we have that
    \begin{equation}
         V_d^0(\widehat{x}^+) - V_d(\widehat{x}^++Gw,\tilde{\theta}^+(w))
         \leq L(\widehat{x}^+,\tilde{\theta}^+(0),\tilde{\bfSigma}^+) - L(\widehat{x}^++Gw,\tilde{\theta}^+(w),\tilde{\bfSigma}^+) \,.\label{eq:V_diff}
    \end{equation}

    We now define
        $\overline{\Sigma} := \arg\max_{\Sigma\in\bbB_{d}}\tr(G'PG\Sigma)$
    and choose $\overline{\dsQ}\in\calP_d$ such that $\ev_{\overline{\dsQ}}[ww']=\overline{\Sigma}$. We combine \cref{eq:L_diff}, \cref{eq:V_diff}, and take the expected value with respect to $\overline{\dsQ}$:
    \begin{align*}
        & V_d^0(\widehat{x}^+) - \ev_{\overline{\dsQ}}\left[V_d(\widehat{x}^++Gw,\tilde{\theta}^+(w))\right] \\
        & \quad = -\tr(G'PG\overline{\Sigma}) - \ev_{\overline{\dsQ}}\left[|Zw-\bfK_f\bfA Gw - \bfK_f\bfB Zw|^2_{\bfS}\right] \\
        & \quad \leq -\delta_d \,.
    \end{align*}
    We combine this inequality with \cref{eq:costdec_xhat} to give
    \begin{equation}\label{eq:Vopt_costdec}
        V_d^0(\widehat{x}^+) \leq V_d^0(x) - \ell(x,v(0)) \leq V_d^0(x) - c_3|x|^2,
    \end{equation}
    in which $c_3>0$ because $Q\succ 0$. Note that the choice of $x\in\calX$ was arbitrary and therefore this inequality holds for all $x\in\calX$. Next, we define the Lyapunov function 
    \begin{equation*}
        H(x) := V_d^0(x) - \max_{\bfSigma\in\bbB_{d}^N}\tr(\bfP\bfSigma),
    \end{equation*}
    in which $\bfP:=I_N\otimes (G'PG)$ from \cref{eq:Vopt_lb}. Note that $V_d(x,\theta)$ is convex by Danskin's Theorem (see \cite[Prop.\,A.3.2]{ref:bertsekas2009convex}). Thus, $V_d^0(x)$ is the partial minimization of a convex function and also convex (see \cite[Prop.\,3.3.3]{ref:bertsekas2009convex}). Therefore, $H:\calX\rightarrow\bbR_{\geq 0}$ is convex because $V_d^0(x)$ is convex. From \cref{eq:Vopt_costdec}, \cref{eq:Vopt_lb}, and \cref{lem:upperbound}, there exist $c_1,c_2,c_3>0$ such that $c_1|x|^2\leq H(x)\leq c_2|x|^2$ and $H(\widehat{x}^+) \leq H(x) - c_3|x|^2$
    Since $H(x)$ is a convex Lyapunov function, $x^+=\widehat{x}^++Gw$, and $\calX$ is compact with the origin in its interior, we have from \citep[Prop. 4.13]{goulart:2007} that \cref{eq:cl} is ISS for any $d\in\calD$. 
\end{proof}

\subsection{Exact long-term performance}

For the class of disturbances in $\calQ$, \citet{munoz:cannon:2020} established that ISS systems converge to the minimal RPI set for the system with probability one. By \cref{as:term_lqr}, the terminal set must contain the minimal RPI set for the system. Thus, we have the following result adapted from \citep[Thm. 5]{munoz:cannon:2020}. 

\begin{lemma}[Convergence to terminal set]\label{lem:bc}
    If \cref{as:convex,as:cost,as:track,as:term_lqr} hold, then for all $\dsP\in\calQ$, $d\in\calD$, and $x\in\calX$, there exists $p\in[0,\infty)$ such that
    \begin{equation*}
        \sum_{k=0}^{\infty}\dsP\big(\phi_d(k;x,\bfw_{\infty})\notin\bbX_f\big) \leq p \,.
    \end{equation*}
\end{lemma}

From \cref{lem:K_f} and the Borel-Cantelli lemma, \cref{lem:bc} implies that for all $x\in\calX$, we have
\begin{equation*}
    \dsP\left(\lim_{k\rightarrow\infty} \phi(k;x,\bfw_{\infty})\in\bbX_f\right) = 1 \,.
\end{equation*}
In other words, the state of the closed-loop system converges to the terminal set $\bbX_f$ with probability one. 
Once in $\bbX_f$, the closed-loop state remains in this terminal set by applying the fixed control law $K_fx$ for all subsequent time step (\cref{lem:K_f}). The long-term performance of the closed-loop system is therefore determined by the control law $K_f$ and, by definition, the matrix $P$ from the DARE in \cref{eq:dare}. We now prove \cref{thm:performance_opt} by formalizing these arguments. 

\begin{proof}[Proof of \cref{thm:performance_opt}]
    Choose $d\in\calD$, $x\in\calX$, and $\dsP\in\calQ$. Define $x(i)=\phi_d(i;x,\bfw_{\infty})$, $u(i)=\kappa_d(x(i))$, $\Sigma=\ev_{\dsP}\left[w(i)w(i)'\right]$, and
    \begin{equation*}
        \zeta(i) := |x(i+1)|_P^2 - |x(i)|_P^2 + \ell(x(i),u(i)) - |Gw(i)|_P^2 \,.
    \end{equation*}
    Recall that $x(i)\in\calX$ for all $i\geq 0$ because $\calX$ is RPI. From \cref{lem:K_f}, we have that if $x(i)\in\bbX_f$, $u(i)=K_fx(i)$ and therefore $\ev_{\dsP}\left[\zeta(i)\right] = 0$. Therefore, we have
    \begin{equation*}
        \ev_{\dsP}\left[\zeta(i)\right] = \ev_{\dsP}\left[\zeta(i)\mid x(i)\notin\bbX_f\right]\dsP\left(x(i)\notin\bbX_f\right) \,.
    \end{equation*}
    Since $\calX$, $\bbU$, and $\bbW$ are bounded, there exists $\eta\geq 0$ such that $|\zeta(i)|\leq \eta$. Thus, 
    \begin{equation*}
        -\eta\dsP\left(x(i)\notin\bbX_f\right)\leq \ev_{\dsP}\left[\zeta(i)\right] \leq \eta\dsP\left(x(i)\notin\bbX_f\right)\,.
    \end{equation*}

    By definition, it holds that
    \begin{equation*}
        \ev_{\dsP}\left[\ell(x(i),u(i)) \right] = \ev_{\dsP}\left[|x(i)|_P^2 - |x(i+1)|_P^2\right] + \ev_{\dsP}\left[\zeta(i)\right] + \tr(G'PG\Sigma) \,.
    \end{equation*}
    We sum both sides from $i=0$ to $k-1$, divide by $k\geq 1$, and rearrange to give
    \begin{equation}\label{eq:sum_delta}
        \calJ_k(x,d,\dsP) = \frac{|x(0)|^2_P - \ev_{\dsP}\left[|x(k)|_P^2\right]}{k} + \frac{1}{k}\sum_{i=0}^{k-1}\ev_{\dsP}\left[\zeta(i)\right] + \tr(G'PG\Sigma) \,.
    \end{equation}
    
    We apply the upper bound on $\ev_{\dsP}\left[\zeta(i)\right]$ in \cref{eq:sum_delta} and note that $|x(k)|_P^2\geq 0$ to give
    \begin{equation*}
        \calJ_k(x,d,\dsP) \leq \frac{|x(0)|^2_P}{k} + \frac{\eta}{k}\sum_{i=0}^{k-1}\dsP\left(x(i)\notin\bbX_f\right) + \tr(G'PG\Sigma) \,.
    \end{equation*}
    From \cref{lem:bc}, there exists $p\in[0,\infty)$ such that
    \begin{equation*}
        \calJ_k(x,d,\dsP) \leq \frac{|x(0)|^2_P+\eta p}{k} + \tr(G'PG\Sigma),
    \end{equation*}
    and therefore\begin{equation}\label{eq:limsup}
        \limsup_{k\rightarrow \infty} \calJ_k(x,d,\dsP)\leq \tr(G'PG\Sigma)\,.
    \end{equation}

    We apply the lower bound for $\ev_{\dsP}\left[\zeta(i)\right]$ in \cref{eq:sum_delta} and note that $|x(0)|_P^2\geq 0$ to give
    \begin{equation*}
        \calJ_k(x,d,\dsP) \geq \frac{-\ev_{\dsP}\left[|x(k)|^2_P\right]}{k} - \frac{\eta}{k}\sum_{i=0}^{k-1}\dsP\left(x(i)\notin\bbX_f\right) + \tr(G'PG\Sigma) \,.
    \end{equation*}
    From \cref{lem:bc}, there exists $p\in[0,\infty)$ such that
    \begin{equation*}
       \calJ_k(x,d,\dsP)\geq \frac{-\ev_{\dsP}\left[|x(k)|^2_P\right]- \eta p}{k} + \tr(G'PG\Sigma) \,.
    \end{equation*}
    Since $\ev_{\dsP}\left[|x(k)|^2_P\right]$ is bounded from \cref{cor:mss} we have
    \begin{equation}\label{eq:liminf}\liminf_{k\rightarrow\infty}\calJ_k(x,d,\dsP) \geq \tr(G'PG\Sigma) \,.
    \end{equation}

    We combine \cref{eq:limsup} and \cref{eq:liminf} to give \cref{eq:performance_optK}. 
    Since the choice of $d\in\calD$, $\dsP\in\calQ$, and $x\in\calX$ was arbitrary, this equality holds for all $d\in\calD$, $\dsP\in\calQ$, and $x\in\calX$. 
\end{proof}

\section{Scalable algorithms}\label{s:algorithms}

We assume for the subsequent discussion that $\varepsilon >0$ and $\theta$ is in a vectorized form, i.e., $\bfM$ is converted to a vector. We first present an exact reformulation of the DRO problem in \cref{eq:dro} and then describe the proposed NT algorithm. The Frank-Wolfe algorithm and additional details regarding the step-size are provided in \cref{s:fw}. 

\subsection{Exact reformulation}

Using existing results in \citep[Prop. 2.8]{nguyen:kuhn:mohajerin:2022} and \citep[Thm. 16]{kuhn:et-al:2019}, we provide an exact reformulation of \cref{eq:dro} via linear matrix inequalities (LMIs) to serve as a baseline for the NT algorithm.

\begin{proposition}[Exact LMI reformulation]\label{prop:sdp}
    Let \cref{as:convex,as:cost} hold and $x\in\calX$. For any $\varepsilon>0$ and $\widehat{\Sigma}\succeq 0$, the min-max problem in \cref{eq:dro} is equivalent to the program
    \begin{align}
        \inf_{\bfM,\bfv,Z,Y,\gamma} \ & |H_x x+ H_u\bfv|^2 + \sum_{k=0}^{N-1}\left(c\gamma_k + \tr(Y_k)\right) \label{eq:lmi}\\
        \textnormal{s.t. } & \gamma_k I \succeq Z_k \quad  \forall \ k\in\{0,\dots,N-1\} \nonumber \\
        & \begin{bmatrix}
            Y_k & \gamma_k\widehat{\Sigma}^{1/2} \\
            \gamma_k\widehat{\Sigma}^{1/2} & \gamma_k I - Z_{k}
        \end{bmatrix} \succeq 0  \quad  \forall \ k\in\{0,\dots,N-1\} \nonumber \\
        & \begin{bmatrix}
            Z & (H_u\bfM + H_w)' \\
            (H_u\bfM + H_w) & I
        \end{bmatrix} \succeq 0 \nonumber \\
        & (\bfM,\bfv)\in\Pi(x), \nonumber
    \end{align}
    in which $c=\varepsilon^2-\tr(\widehat{\Sigma})$ and $Z_{k}\in\bbR^{q\times q}$ is the $k^{\rm th}$ block diagonal of $Z$.    
\end{proposition}

\begin{proof}
    Define $\tilde{Z}(\theta)=(H_u\bfM+H_w)'(H_u\bfM+H_w)\succeq 0$ and
    \begin{equation*}
         V_d(x,\theta) = \max_{\bfSigma\in\bbB_d^N}\tr\Big(\tilde{Z}(\theta)\bfSigma\Big) = \min_{Z\succeq \tilde{Z}(\theta)}\max_{\bfSigma\in\bbB_d^N}\tr\Big(Z\bfSigma\Big)\,.
    \end{equation*}
    From the structure of $\bfSigma\in\bbB_d^N$, we have
    \begin{equation}
        \max_{\bfSigma\in\bbB_d^N}\tr\Big(Z\bfSigma\Big) = \sum_{k=0}^{N-1}\max_{\Sigma_k\in\bbB_d}\tr\Big(Z_k\Sigma_k\Big), \label{eq:sum_Sigma}
    \end{equation}
    in which $Z_k\in\bbR^{q\times q}$ is the $k$-th block diagonal of $Z$. From \citep[Thm. 16]{kuhn:et-al:2019} and \citep[Prop. 2.8]{nguyen:kuhn:mohajerin:2022}, we can write the dual of $\max_{\Sigma_k\in\bbB_d}\tr(Z_k\Sigma_k)$ via an LMI.  
    Substituting this dual formulation into \cref{eq:sum_Sigma} and reformulating $Z \succeq (H_u\bfM+H_w)'(H_u\bfM+H_w)$ via the Schur complement gives \cref{eq:lmi}. 
\end{proof}

If $\bbZ$, $\bbX_f$, and $\bbW$ are polytopes, then this reformulation can be solved as an LMI optimization problem with standard software such as MOSEK \citep{mosek:2022}. While this LMI optimization problem can be solved quickly and reliably for small problems, larger problems are unfortunately not practically scalable compared to the usual QPs in linear MPC formulations. 

\begin{remark}[Saddle point]\label{rem:saddle}
    The set of saddle points~$(\theta^*,\bfSigma^*)$ for the min-max problem~\cref{eq:dro_L} is non-empty and compact for any $x\in\calX$ and $d\in\calD$. This fact is a classical result for the convex-concave function $L(x,\cdot)$ ensured by the convexity and compactness of $\Pi(x)$ (\cref{lem:compact}) and $\bbB_d^N$ \citep[Lemma A.6]{nguyen:shafieezadeh:kuhn:mohajerin:2023}, see for instance \cite[Prop.\,5.5.7]{ref:bertsekas2009convex}.
\end{remark}

\subsection{Newton-type saddle point algorithm}
We introduce a new algorithm that exploits the structure of the min-max problem~\eqref{eq:dro_L} through a Newton-type step. Additional information regarding the stepsize is provided in \cref{s:fw}. Code to implement this algorithm is available in the GitHub repository~\url{https://github.com/rdmcallister/drmpc}. 

We subsequently assume that $\widehat{\Sigma}\succ 0$. For notational simplicity, given any (fixed)~$x\in\calX$ and $d\in\calD$, we use the shorthand notation for the constraint $\theta\in\Pi:=\Pi(x)$ and objective function~$L_x(\cdot)=L(x,\cdot)$ in~\eqref{eq:L_obj} to rewrite the min-max~\eqref{eq:dro_L} concisely as
\begin{equation}
\label{eq:innermax}
    \min_{\theta\in\Pi} \bigg\{f(\theta) := \max_{\bfSigma\in\bbB_d^N} L_x(\theta,\bfSigma)\bigg\} \,.
\end{equation}
The structure of the set~$\bbB_d^N$ in \eqref{B^N_d} allows us to expand \cref{eq:innermax}~as 
\begin{equation}\label{eq:innermax_iid}
    f(\theta) = L_x(\theta,\mathbf{0}) + \sum_{k=0}^{N-1}\max_{\Sigma_k\in\bbB_d}\tr\big(Z_k(\theta)\Sigma_k\big),
\end{equation}
where the matrix $Z_k(\theta)\in\bbR^{q\times q}$ is the $k$\textsuperscript{th} block diagonal of $(H_u\bfM+H_w)'(H_u\bfM+H_w)$. Each maximization in \cref{eq:innermax_iid} can be solved in finite time using a bisection algorithm detailed in \cite[Alg. 2, Thm. 6.4]{nguyen:shafieezadeh:kuhn:mohajerin:2023}. Hence, we have access to
\begin{equation}\label{eq:innermax_k}
    \bfSigma^*(\theta) := \arg\max_{\bfSigma\in\bbB_d^N} L_x(\theta,\bfSigma),
\end{equation}
% . With this bisection algorithm, we can also compute the optimal solution
% \begin{equation}\label{eq:innermax_k}
%     \Sigma_k^*(\theta) := \arg\max_{\Sigma_k\in\bbB_d}\tr\big(Z_k(\theta)\Sigma_k\big)
% \end{equation}
% and construct the block-diagonal matrix
% \begin{equation*}
%     \bfSigma^*(\theta) = \textnormal{diag}\left(\begin{bmatrix}
%         \Sigma_0^*(\theta) &\cdots &\Sigma_{N-1}^*(\theta)
%     \end{bmatrix}\right)
% \end{equation*}
% and the DRO problem in \cref{eq:dro_L} is then equivalent to
% \begin{equation}\label{eq:opt}
%     f^* := \min_{\theta\in\Pi}f(\theta)\,.
% \end{equation}
Since the solution to \cref{eq:innermax_k} is unique for $\widehat{\Sigma}\succ 0$~\citep[Prop. A.2]{nguyen:shafieezadeh:kuhn:mohajerin:2023}, we have from Danskin's theorem that $f(\theta)$ is convex and $\nabla f(\theta) = \nabla_{\theta} L_x(\theta, \bfSigma^0(\theta))$, i.e., the gradient of $f$ at $\theta$ is given by the gradient of $L_x(\cdot)$ with respect to $\theta$, evaluated at $(\theta,\bfSigma^0(\theta))$. Frank-Wolfe (FW) algorithms that exploit this property of the min-max program in \cref{eq:innermax} have shown promising results for DRO problems (e.g., \citep{nguyen:shafieezadeh:kuhn:mohajerin:2023,sheriff:mohajerin:2023}). However, FW algorithms are often limited to sublinear convergence rates for MPC optimization problems (see~\cref{fig:convergence} in the numerical section) because the constraint set $\Pi$ is not strongly convex (e.g., polytope) and the minimizer is frequently on the boundary of $\Pi$. Thus, we propose a Newton-type algorithm in which the search direction is determined by using a second-order approximation of~$f$ around a given value of $\theta$ and $\bfSigma^*(\theta)$:
\begin{equation}\label{eq:taylor}
    \tilde{f}(\vartheta,\theta) :=  \nabla_{\theta}L_x(\theta,\bfSigma^*(\theta))'(\vartheta-\theta) \\ + \frac{1}{2}(\vartheta-\theta)'\nabla_{\theta\theta}L_x(\theta,\bfSigma^*(\theta))(\vartheta-\theta)\,.
\end{equation}
With the above approximation, at each iteration we solve 
\begin{align}
\label{2-oracle}
    F(\theta) := \arg\min_{\vartheta\in\Pi} \tilde{f}(\vartheta,\theta)\,.
\end{align}
When $\Pi$ is a polytope, the oracle~\eqref{2-oracle} is a QP. Recall that our objective function~$L_x$ in~\eqref{eq:innermax} (originally defined in~\eqref{eq:L_obj}) is a quadratic function in $\theta$, and therefore $F(\theta)=\arg\min_{\vartheta\in\Pi} L_x(\vartheta,\bfSigma^*(\theta))$. The solution in \cref{2-oracle} defines the search direction for the iteration through the update rule
\begin{equation}\label{eq:nt_update}
   \theta_{t+1} = \theta_{t} + \eta_t\left(F(\theta_t) - \theta_t\right),
\end{equation}
in which the stepsize~$\eta_t$ is chosen according to either the adaptive or fully adaptive step-size rules %used in FW algorithms. The adaptive stepsize is 
determined by
\begin{equation}\label{eq:adaptive}
    \eta_t(\beta) = \min\left\{1, \frac{(\theta_t-F(\theta_t))'\nabla f(\theta_t)}{\beta |\theta_t-F(\theta_t)|^2}\right\} \,.
\end{equation}
The parameter~$\beta$ in \eqref{eq:adaptive} is a conservative estimate of the \emph{global} smoothness parameter for the function~$f$ in \eqref{eq:innermax}. The fully adaptive stepsize, first proposed in \cite{pedregosa:negiar:askari:jaggi:2020}, instead determines a local value of $\beta$ via backtracking line search and a quadratic sufficient decrease condition. 
In~\cref{alg:NT}, we summarize the NT algorithm with fully adaptive stepsize in pseudocode.

\begin{center}
\begin{algorithm}
    \caption{NT algorithm (fully adaptive stepsize)}
    \SetKwInOut{Output}{Output}
    \SetKwInOut{Input}{Input}
    \Input{Initial $\theta_0\in\Pi$, smoothness parameter $\beta_{-1}>0$, line search parameters $\zeta,\tau>1$}
    set $t\leftarrow 0$ \\
    \While{stopping criterion not met}{
    solve $\tilde{\theta}_t = F(\theta_t)$ \\
    set $d_t\leftarrow\tilde{\theta}_t - \theta_t$ and $g_t\leftarrow -d_t'\nabla f(\theta_t)$ \\
    set $\beta_t \leftarrow \beta_{t-1}/\zeta$ and $\eta_t \leftarrow \min\{1,g_t/(\beta_t|d_t|^2)\}$ \\
    \While{$f(\theta_t+\eta d_t) > f(\theta_t) - \eta_t g_t + (\eta_t^2\beta_t/2)|d_t|^2$}
    {
    set $\beta_t \leftarrow \tau\beta_t$ and $\eta_t\leftarrow \min\{1,g_t/(\beta_t|d_t|^2)\}$
    }
    set $\theta_{t+1}\leftarrow \theta_t + \eta_td_t$ \\
    set $t\leftarrow t+1$
    }
    \Output{$\theta_t$}
    \label{alg:NT}
\end{algorithm}
\end{center}

\section{Numerical Examples}
\label{s:examples}
%=========================
We present two examples. The first is a small-scale (2 state) example that is used to demonstrate the closed-loop performance guarantees presented in \cref{s:cl} and investigate the computational performance of the proposed NT algorithm. The second is a large-scale (20 state) example, based on the Shell oil fractionator case study in \citet[s. 9.1]{maciejowski:2001}, that is used to demonstrate the scalability of the proposed NT algorithm. All optimization problems (LP, QP, or LMI) are solved with MOSEK with default parameter settings \citep{mosek:2022}. Code to generate all figures for the small-scale example is available in the GitHub repository~\url{https://github.com/rdmcallister/drmpc}. 

\subsection{Small-scale example}

We consider a two-state, two-input system in which 
\begin{equation*}
    A = \begin{bmatrix}
        0.9 & 0 \\
        0.2 & 0.8
    \end{bmatrix} \qquad 
    B = G = \begin{bmatrix}
        1 & 0 \\ 0 & 1
    \end{bmatrix}\,.
\end{equation*}
We define the input constraints $\bbU:=\{u\in\bbR^2\mid |u|_{\infty}\leq 1, \ u_2\geq 0\}$ and note that the origin is on the boundary of $\bbU$. We consider the disturbance set $\bbW:=\{w\in\bbR^2\mid |w|_{\infty}\leq 1\}$ with the nominal covariance 
\begin{equation*}
\widehat{\Sigma} = \begin{bmatrix}
0.01 & 0 \\ 0 & 0.01
\end{bmatrix}
\end{equation*}
and $\varepsilon=0.1$ for DRMPC. We choose
\begin{equation*}
    Q = \begin{bmatrix}
        0.1 & 0 \\ 0 & 10
    \end{bmatrix}
    \qquad 
    R = \begin{bmatrix}
        10 & 0 \\ 0 & 0.1
    \end{bmatrix}
\end{equation*}
and define $P\succ 0$ as the solution to the the Lyapunov equation for this system with $u=0$, i.e., $P\succ 0$ satisfies $A'PA - P = -Q$. This DRMPC formulation satisfies \cref{as:convex,as:cost,as:term,as:track}, but \textit{not} \cref{as:term_lqr}. Recall from \cref{cor:equal} that DRMPC formulations that satisfy \cref{as:term_lqr} will produce the same long-term performance as SMPC even if SMPC has an incorrect disturbance distribution.

\subsubsection*{Computational performance}

We solve the DRMPC problem for this formulation using three different methods: 1) the LMI optimization problem in \cref{eq:lmi} with MOSEK, 2) the FW algorithm, and 3) the proposed NT algorithm in \cref{alg:NT}. To compare these algorithms, we use a fixed initial condition of $x_0=\begin{bmatrix} 1 & 1
\end{bmatrix}'$ and horizon length $N=10$. We plot the convergence rate in terms of suboptimality gap ($f(\theta_t)-f^*$) for the FW and NT algorithms in \cref{fig:convergence}. For this suboptimality gap, we determine $f^*$ via the LMI optimization problem in \cref{eq:lmi}.
% Therefore, convergence in the suboptimality gap implies that the FW/NT algorithm converge to the same optimal cost as the exact LMI reformulation. 
For both FW and NT algorithms, we consider the adaptive (A) and fully adaptive (FA) step-size rules. We terminate when the duality gap is less than $10^{-6}$.  

First, we discuss the per-iteration convergence rate shown in the top of \cref{fig:convergence}. For the FW algorithm, the convergence rate is \emph{sublinear} for both step-size rules and does not converge within $10^3$ iterations.
% \footnote{Even for RMPC, i.e., $d=(0,0)$, we observe only sublinear convergence for the FW algorithm.} 
% Moreover, the suboptimality gaps of the FW algorithms remains significantly larger than the specified tolerance of $10^{-6}$ after $10^3$ iterations.  
By contrast, the NT algorithm appears to obtain a \emph{superlinear} (perhaps quadratic) convergence rate near the optimal solution. This behavior is also observed for all other values of the initial condition and horizon length investigated. In fact, the fully adaptive NT algorithm typically converges in fewer than $5$ iterations. The significant improvement in per-iteration convergence rate ensures that the NT algorithm requires less computation time than FW despite solving a QP at each iteration (bottom plot of \cref{fig:convergence}). 

%TODO: Update!
\begin{figure}[ht]
    \centering
    \includegraphics{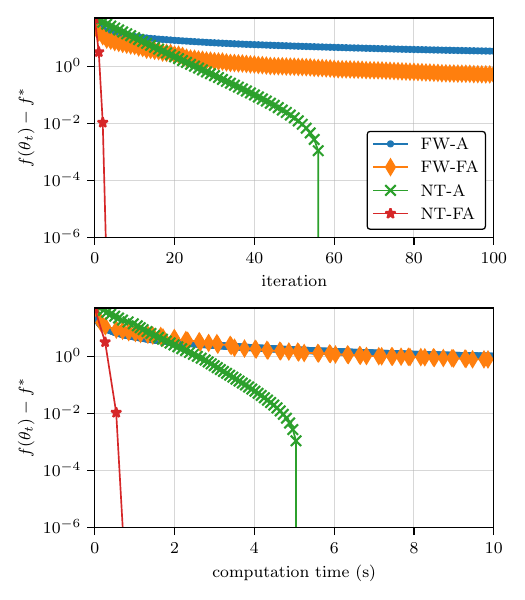}
    \caption{Convergence of FW and NT algorithms with adaptive (A) and fully adaptive (FA) stepsize calculations for a DRMPC problem ($N=10$) in terms of suboptimality gap as a function of iteration (left) and computation time (right).}
    \label{fig:convergence}
\end{figure}

In \cref{fig:comparison}, we compare the computation times required to solve the small scale DRMPC problem for difference horizon lengths $N$ via 1) the LMI optimization problem in \cref{eq:lmi} with MOSEK and 2) the NT algorithm.  
For $N\leq 5$, and therefore fewer variables and constraints, solving the DRMPC problem as an LMI optimization problem is faster. For $N>5$, however, the NT algorithm is faster than the LMI formulation. For $N\geq 15$, the computation time to solve the DRMPC problem as an LMI optimization problem is more than \textit{twice} the computation time required for the NT algorithm. 

\begin{figure}
    \centering
    \includegraphics{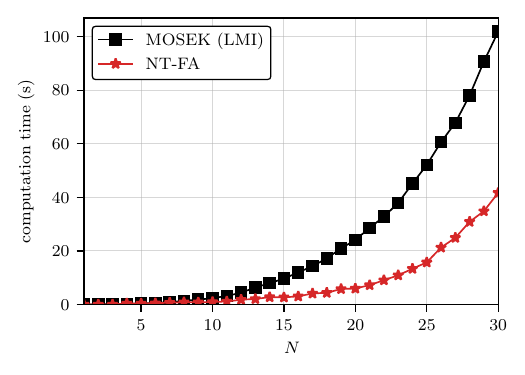}
    \caption{Comparison of computation times for the NT algorithm (fully adaptive stepsize) and LMI formulation solved by MOSEK for the horizon $N$.}
    \label{fig:comparison}
\end{figure}

\subsubsection*{Closed-loop Performance}

We initialize the state at $x(0)=\begin{bmatrix}1 & 1
\end{bmatrix}'$ and use $N=10$. In the following discussion, we consider three different controllers: DRMPC with $d=(\varepsilon,\widehat{\Sigma})$, SMPC with $d=(0,\widehat{\Sigma})$, and RMPC with $d=(0,0)$.

To demonstrate some differences between DRMPC, SMPC, and RMPC, we plot the first element of closed-loop state trajectory assuming the disturbance is zero, i.e., $w=0$, in \cref{fig:n2_nom}. RMPC drives the closed-loop state to the origin. SMPC, however, does not drive the closed-loop state to the origin even though the disturbance is zero. Since $u_2\geq 0$, the SMPC controller keeps $x_1$ slightly below the origin to mitigate the effect of positive values for $w_2$. The amount of offset is determined by the covariance of the disturbance. Since DRMPC considers a worst-case covariance for the disturbances, the offset is larger. Thus, for $||\bfw_{\infty}||=0$, the closed-loop state for DRMPC (SMPC) does not converge to the origin. The origin is therefore not ISS for DRMPC (SMPC), despite satisfying \cref{as:convex,as:cost,as:term,as:track}. By contrast, these assumptions render the origin ISS for the closed-loop system generated by RMPC \citep[Thm. 23]{goulart:kerrigan:maciejowski:2006}. 
To summarize: SMPC and DRMPC are hedging against uncertainty and thereby giving up the deterministic properties of RMPC, such as ISS, in the pursuit of improved performance in terms of the expected value of the stage cost, i.e., $\calJ_k(\cdot)$.  

\begin{figure}
    \centering
    \includegraphics{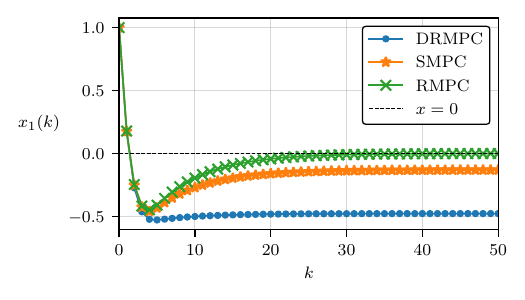}
    \caption{Closed-loop trajectories ($N=10$) with zero disturbance, i.e., $x(k)=\phi_d(k;x,\mathbf{0})$, for the first element of the state, denoted $x_1(k)$.}
    \label{fig:n2_nom}
\end{figure}

We now investigate the performance of DRMPC relative to SMPC/RMPC for a distribution $\dsP\in\calP_d^{\infty}$. Specifically, we consider $w(k)$ to be i.i.d. in time and defined as $w(k)=\Sigma^{1/2} \omega(k)$, in which $\omega_1(k),\omega_1(k)$ are independently sampled from a uniform distribution between $[-\sqrt{3},\sqrt{3}]$ and
\begin{equation*}
    \Sigma = \begin{bmatrix}
        0.01 & 0.01 \\
        0.01 & 0.035
    \end{bmatrix}\,.
\end{equation*}
Thus, $w(k)\in\bbW$, is zero mean, and has a covariance of $\Sigma$.
Note that the covariance $\widehat{\Sigma}$ used in the SMPC formulation is different than the covariance of the disturbance encountered in the closed-loop system, i.e., $\widehat{\Sigma}\neq\Sigma$.

We simulate $S=100$ different realizations of the disturbance trajectory for each controller. For each simulation $s\in\{1,\dots,S\}$, we define the closed-loop state and input trajectory $x^s(k)$ and $u^s(k)$, as well as the time-average cost
\begin{equation*}
    \calJ_k^s:=\frac{1}{k}\sum_{i=0}^{k-1}\ell(x^s(k),u^s(k))\,.
\end{equation*}
In accordance with the results in \cref{thm:performance,cor:mss}, we consider the sample average approximations of $\ev_{\dsP}\left[|\phi(k;x,\bfw_{\infty})|\right]$ and $\calJ_k(x,d,\dsP)$ defined as
\begin{equation*}
    \tilde{\ev}_{\dsP}\left[|x(k)|^2\right] := \frac{1}{S}\sum_{s=1}^{S} |x^s(k)|^2 \qquad \tilde{\calJ}_k:= \frac{1}{S}\sum_{s=1}^{S}\calJ_k^s\,.
\end{equation*}

In \cref{fig:n2_stats}, we plot $\tilde{\ev}_{\dsP}\left[|x(k)|^2\right]$ and $\tilde{\calJ}_k$. For each algorithm, we observe an initial, exponential decay in the mean-squared distance $\tilde{\ev}_{\dsP}\left[|x(k)|^2\right]$ towards a constant, but nonzero, value. These results for DRMPC are consistent with \cref{cor:mss}. We note, however, that DRMPC produces the largest value of $\tilde{\ev}[|x(k)|^2]$, i.e., the mean-squared distance between the closed-loop state and the setpoint is larger for DRMPC than for SMPC or RMPC. While this result may initially seem counter-intuitive, the objective prescribed to the DRMPC problem is to minimize the expected value of the stage cost, not the expected distance to the origin. In terms of the expected value of the stage cost, i.e., $\tilde{\calJ}_k$, the performance of DRMPC is better than SMPC, which is better than RMPC. This difference becomes more pronounced as $k\rightarrow\infty$. 

\begin{figure}
    \centering
    \includegraphics{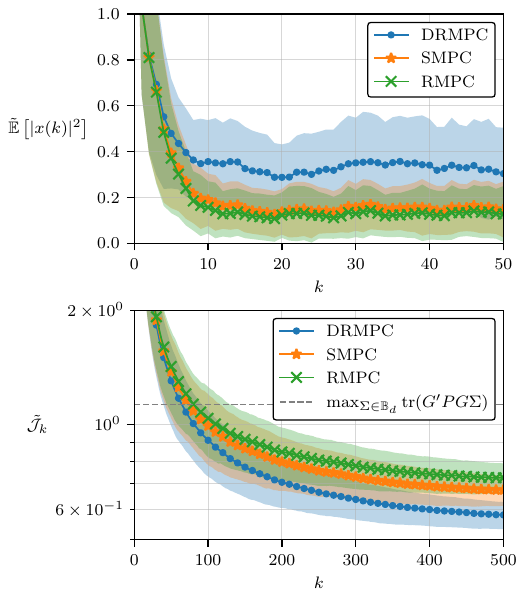}
    \caption{Sample averages of $\ev_{\dsP}\left[|\phi(k;x,\bfw_{\infty})|^2\right]$ and $\calJ_k(x,d,\dsP)$, denoted $\tilde{\ev}_{\dsP}\left[|x(k)|^2\right]$ and $\tilde{\calJ}_k$, for $S=100$ realizations of the closed-loop trajectory. Shaded regions show plus/minus one standard deviation.}
    \label{fig:n2_stats}
\end{figure}

We note that $\Sigma\in\bbB_d$ in this example is intentionally chosen to exacerbate the effect of the disturbance on $x_2$ and thereby increase the cost of the closed-loop trajectory, i.e., a worst-case distribution. Therefore, DRMPC produces a superior control law relative to SMPC. If the ambiguity set, however, becomes too large relative to this value of $\Sigma$, the additional conservatism of DRMPC can produce worse performance than SMPC in terms of $\tilde{\calJ}_k$ for a fixed value of $\Sigma$. To demonstrate this tradeoff, we consider the same closed-loop simulation and plot the value of $\tilde{\calJ}_T$ at $T=500$ for various values of $\varepsilon$ and fixed $\widehat{\Sigma}$. 
In \cref{fig:n2_rhos}, we observe that $\varepsilon\approx 0.11$ achieves the minimum value of $\calJ_T$, with an approximately 13\% decrease in the value of $\tilde{\calJ}_T$ compared to $\varepsilon=0.01$. For values of $\varepsilon> 0.11$, the value of $\tilde{\calJ}_T$ increases significantly until leveling off around $\varepsilon=1$. For large values  of $\varepsilon$, DRMPC is too conservative because $\Sigma$ is now well within the interior of $\bbB_d$. 

Another interesting feature in \cref{fig:n2_rhos} is that the range of values for $\calJ_T^s$ for each $s\in\{1,\dots,30\}$, shown by the shaded region, decreases as $\varepsilon$ increases. This behavior might also be explained by the increased conservatism of DRMPC as $\varepsilon$ increases. As the value of $\varepsilon$ increase, 
% DRMPC generates a closed-loop system that is distributionally robust to a range of possible covariances. In this case, 
we might expect DRMPC to drive the closed-loop system to an operating region that attenuates the effect of all disturbances on the closed-loop cost at the expense of nominal performance. Thus, the closed-loop system becomes less sensitive to disturbances and thereby decreases the variability in performance at the expense of an increase in average performance.
In summary, we are left with the classic trade-off in robust controller design; If we choose $\varepsilon$ too large, the excessively conservative DRMPC may perform worse than SMPC ($\varepsilon=0$) even if the disturbance distribution used for SMPC is incorrect. Thus, the design goal for DRMPC is to select a value of $\varepsilon$ that balances these two extremes.

\begin{figure}
    \centering
    \includegraphics{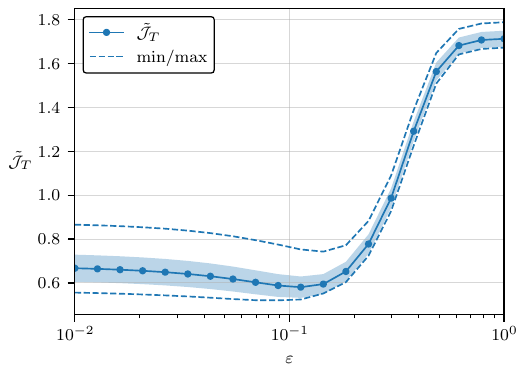}
    \caption{Sample average of the performance metric $\calJ_T(x,d,\dsP)$, denoted $\tilde{\calJ}_T$, for $T=500$ as a function of the ambiguity radius $\varepsilon$ for $S=30$ realizations of the disturbance trajectory. The shaded region shows plus/minus one standard deviation and the dashed lines indicate the min/max values of $\tilde{\calJ}_T^s$ for all $s$.}
    \label{fig:n2_rhos}
\end{figure}

\subsection{Large scale example: Shell oil fractionator}

To demonstrate the applicability of the Newton-type algorithm to control problems of an industrially relevant size, we now consider the Shell oil fractionator example in \citet[s. 9.1]{maciejowski:2001} with $n=20$ states, $m=3$ inputs, and $p=3$ outputs. We include two disturbances ($q=2$): the intermediate reflux duty and the upper reflux duty.  

We consider $\bbU:= \left\{u\in\bbR^3\mid |u|_{\infty}\leq 1, \ u_3\geq 0\right\}$ in which the origin is again on the boundary of the input constraint $\bbU$. The disturbance set is $\bbW:=\{w\in\bbR^2\mid |w|_{\infty}\leq 1\}$ with the nominal covariance $\widehat{\Sigma}=0.01I_2$
and ambiguity an radius of $\varepsilon=0.1$. 
The outputs $y$ satisfy $y=Cx$ and we define cost matrices as
$Q=C'Q_yC$ and $R=0.1I_3$ in which $Q_y=\textnormal{diag}([20, \ 10, \ 1])$. 
We then define the terminal cost matrix $P\succ 0$ as the solution to the Lyapunov equation $A'PA-P=-Q$ because $A$ is Schur stable. This DRMPC problem formulation satisfies \cref{as:convex,as:cost,as:term} and $(A,Q^{1/2})$ is detectable (See \cref{rem:detectable}). However, this formulation does not satisfy \cref{as:term_lqr} so the performance of DRMPC and SMPC may differ (see \cref{cor:equal}).

We initialize the state at $x(0)=0$ and use $N=10$. We consider the performance of the closed-loop system in which $w(k)$ is sampled from a zero-mean uniform distribution with a covariance of $\Sigma = \textnormal{diag}([0.04, \ 0.01])$. 
Note that $\Sigma\in\bbB_d$ but $\Sigma\neq\widehat{\Sigma}$. We simulate $T=500$ time steps for $S=30$ realizations of the disturbance trajectory. We plot $\tilde{\calJ}_k$ in \cref{fig:oil_stats} for DRMPC, SMPC, and RMPC. At $T=500$, we observe an almost negligible 0.2\% decrease in $\tilde{\calJ}_{T}$ for DRMPC relative to SMPC. The standard deviation of the closed-loop performance is also nearly identical for all three controllers. Longer horizons may increase this difference, but we expect the overall benefit of DRMPC to remain small and therefore not worth the extra computation.   

\begin{figure}
    \centering
    \includegraphics{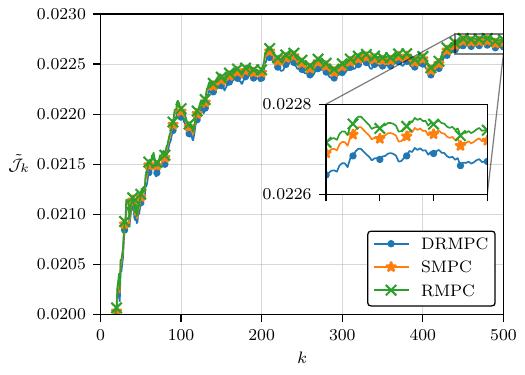}
    \caption{Sample average of $\calJ_k(x,d,\dsP)$, denoted $\widehat{\calJ}_k$, for $S=30$ realizations of the closed-loop trajectory.}
    \label{fig:oil_stats}
\end{figure}

\section{Summary and comparison}

In \cref{tab:comparison}, we summarize some key observations from the theoretical analysis, algorithm development, and numerical examples covered in this work if \cref{as:term} holds, but \cref{as:term_lqr} does not hold. If \cref{as:term_lqr} also holds, then SMPC and DRMPC guarantee ISS (\cref{thm:iss}) but only transient performance benefits are achievable regardless of the value of $\Sigma$ and $\widehat{\Sigma}$ (\cref{cor:equal}). The key characteristics of a control problem that may justify the extra computational expense of DRMPC are: (1) the covariance of the disturbance is large, but not well known, (2) the origin (target steady-state) is near input/state constraints, and (3) performance in terms of stage cost is more important than ISS. 

\def\arraystretch{1.5}
\begin{table}[ht]
    \centering
    \caption{Comparison between R/S/DRMPC with \cref{as:term}.}
    \begin{tabular}{c|ccc}
        MPC type & Preferred if & ISS & Computation \\ \hline
         RMPC & $\Sigma$ is small & Yes & QP \\[2mm] 
         SMPC & \makecell{$\Sigma$ is large, \\ $\widehat{\Sigma}\approx \Sigma$} & No & Harder QP \\[3mm] 
         DRMPC & \makecell{$\Sigma$ is large, \\ $\widehat{\Sigma}$ is poor estimate of $\Sigma$} & No & SDP (several QPs) \\[2mm] \hline
    \end{tabular}
    \label{tab:comparison}
\end{table}

\appendix 

\section{Additional technical proofs and results}
\label{s:appendix}

\begin{proof}[Proof of \cref{lem:compact}]
    From \citet[Thm. 3.5]{goulart:2007} we have that $\Pi(x)$ is closed and convex for all $x\in\calX$ and $\calX$ is closed and convex. We now establish that $\Pi(x)$ is bounded. If $(\bfM,\bfv)\in\Pi(x)$, then $\bfM\bfw + \bfv \in \bbU^N$ for all $\bfw\in\bbW^N$. Since $0\in\bbW$, we have that for any $(\bfM,\bfv)\in\Pi(x)$, $\bfv$ must satisfy $\bfv\in\bbU^N$ 
    Moreover, since the origin is in the interior of $\bbW$, there exists $\delta>0$ such that $B_{\delta}:=\{\bfw\in\bbR^{Nq}\mid |\bfw|\leq\delta \}\subseteq\bbW$. Since $\bbU$ is bounded, there exists $b\geq 0$, such that $|\bfu|\leq b$ for all $\bfu\in\bbU^N$. For any $(\bfM,\bfv)\in\Pi(x)$, we have 
    \begin{equation}\label{eq:M_bound}
        |\bfM\bfw| \leq 2b \quad \forall \ \bfw\in B_{\delta}
    \end{equation}
    because $\bfv\in\bbU^N$. We have that \cref{eq:M_bound} is equivalent to
    \begin{equation*}
         \lVert \bfM\rVert_2 := \sup\left\{|\bfM\bfw| \ \middle| \ |\bfw|\leq 1 \right\} \leq 2b/\delta
    \end{equation*}
    and we can construct a bounded set for $\bfM$ as follows:
    \begin{equation*}
        \bfM \in \mathbb{M} := \left\{ \bfM\in\bbR^{Nm\times Nq} \ \middle| \ ||\bfM||_2 \leq 2b/\delta \right\} 
    \end{equation*}
    for all $(\bfM,\bfv)\in\Pi(x)$. Therefore, $(\bfM,\bfv)\in\Pi(x) \subseteq \mathbb{M} \times \bbU^N$ for all $x\in\calX$. Since $\bbU$ and $\mathbb{M}$ are bounded, $\Pi(x)$ is bounded as well, uniformly for all $x\in\calX$.  
\end{proof}

\begin{lemma}\label{lem:calD}
    If $\bbW$ is bounded, then $\calD$ is bounded. 
\end{lemma}

\begin{proof}
    Define the set $\mathbb{S}:=\{\Sigma = \ev_{\dsP}\left[ww'\right] \mid \dsP\in\calM(\bbW)\}$ and note that $\mathbb{S}$ is bounded because $\bbW$ is bounded. By definition, $\bbB_d\subseteq\mathbb{S}$ for all $d\in\calD$. Since $\widehat{\Sigma}\in \bbB_d$ for all $d=(\varepsilon,\widehat{\Sigma})$, then $\widehat{\Sigma}\in\mathbb{S}$ for all $d\in\calD$ and therefore $\calD\subseteq \bbR_{\geq 0}\times \mathbb{S}$.
    
    Define $\rho:=\sup_{\Sigma\in\mathbb{S}}\tr(\Sigma)^{1/2}<\infty$. Therefore, $\mathbb{S}=\bbB_{(\rho,0)}$. Choose any $\widehat{\Sigma}\in\mathbb{S}$. If $\Sigma\in\bbB_{(\rho,0)}$, then 
    \begin{equation*}
        \tr\left(\widehat{\Sigma} + \Sigma - 2\left(\widehat{\Sigma}^{1/2}\Sigma\widehat{\Sigma}^{1/2}\right)^{1/2}\right) \leq \tr\left(\widehat{\Sigma}\right) + \tr\left(\Sigma\right) \leq 2\rho
    \end{equation*}
    and therefore $\Sigma\in\bbB_{(2\rho,\hat{\Sigma})}$. Thus, $\mathbb{S}\subseteq\bbB_{(\rho,0)}\subseteq \bbB_{(2\rho,\hat{\Sigma})}$. For all $(\varepsilon,\widehat{\Sigma})\in\calD$, we have $\bbB_{(\varepsilon,\widehat{\Sigma})} \subseteq \mathbb{S} \subseteq \bbB_{(2\rho,\hat{\Sigma})}$ and therefore $\varepsilon\leq 2\rho$. Hence, $\calD\subseteq [0,2\rho]\times\mathbb{S}$ and $\calD$ is bounded.
\end{proof}

\section{Optimal terminal cost}\label{app:optP}

Since the term $\tr(G'PG\Sigma)$ appears on the right-hand side of all bounds in \cref{thm:performance,thm:resie,cor:mss}, a straightforward design strategy for DRMPC is to select the value of $P$ that minimizes this term subject to the constraints in \cref{as:term}. In the following lemma, we show that the value of $P$ that achieves this goal is given by \cref{eq:dare} in \cref{as:term_lqr}.  

\begin{lemma}[Optimal terminal cost]\label{lem:optP}
    For any $Q,R\succ 0$, a solution to 
    \begin{align*}
        \inf_{P\succ 0,K_f} \ & \tr(G'PG\Sigma) \\ 
        \textnormal{s.t. } &  P - Q - K_f'RK_f \succeq (A+BK_f)'P(A+BK_f) 
    \end{align*}
    is given by the solution to the DARE \cref{eq:dare}
    and the associated controller $K_f:=-(R+B'PB)B'PA$ for all $G\in\bbR^{m\times q}$ and $\Sigma\succeq 0$.
\end{lemma}

\Cref{lem:optP} suggests that, if possible, one should always select $P$ according to \cref{as:term_lqr} to minimize the value of $\tr(G'PG\Sigma)$ for all $\Sigma\succeq 0$. Therefore, \cref{as:term_lqr}, in addition to being a convenient and common method to select $P$, also produces the best theoretical bound for the performance of the closed-loop system. 

\begin{proof}[Proof of \cref{lem:optP}]
    Consider the unique solution $P^*\succ 0$ that satisfies \cref{eq:dare} and choose any $G\in\bbR^{m\times q}$ and $\Sigma\succeq 0$. We denote $D:=G\Sigma G'$ and note that $D\succeq 0$. We consider the equivalent optimization problem
\begin{subequations}
    \begin{align}
        \inf_{K_f,S} \ & \tr((P^*-S) D) \\ 
        \textnormal{s.t. } &  P^* - S - Q - K_f'RK_f \succeq (A+BK_f)'(P^*-S)(A+BK_f) \label{eq:Lyap_Pstar}\\
        & P^*-S \succ 0, \quad S \succeq 0 \label{eq:pd_Pstar}
    \end{align}
\end{subequations}
    If the solution to this optimization problem is $S=0$ for all $D\succeq 0$, then $P^*$ must be a solution to the original optimization problem. 
    
    We now prove that $S=0$. Assume $K_f\in\bbR^{m\times n}$ and $S\succeq 0$ exist such that \cref{eq:Lyap_Pstar,eq:pd_Pstar} hold. Note that this inequality implies that $A+BK_f$ is Schur stable. Since $P^*$ defines the optimal cost of the infinite horizon unconstrained LQR problem, any $K_f\in\bbR^{m\times n}$ must satisfy
    \begin{equation}\label{eq:Pstar}
        (A+BK_f)'(P^*)(A+BK_f) \succeq P^* - Q - K_f'RK_f
    \end{equation}
    By combining \cref{eq:Lyap_Pstar,eq:Pstar}, we have that $K_f$ and $S\succeq 0$ must satisfy
    \begin{equation}\label{eq:S}
        (A+BK_f)'S(A+BK_f) \succeq S
    \end{equation}
    However, if there exists $S\neq 0$ satisfying \cref{eq:S} with $S\succeq 0$, then $A+BK_f$ is not Schur stable. Since $A+BK_f$ is Schur stable, the only value of $S\succeq 0$ that satisfies \cref{eq:S} is $S=0$. Therefore, $S=0$ is the only solution to the modified optimization problem and $P^*$ is the solution to the original optimization problem. Note that since the choices of $G\in\bbR^{m\times q}$ and $\Sigma\succeq 0$ were arbitrary, this solution holds for any $G\in\bbR^{m\times q}$ and $\Sigma\succeq 0$.     
\end{proof}

\section{Fundamental mathematical properties of DRMPC formulation}\label{app:problem}

We now establish some fundamental mathmatical properties of \cref{eq:dro} to ensure that subsequently defined quantities are indeed well-defined. 

We first introduce some notation and definitions. A function $f:X\rightarrow \bbR$ is lower semincontinuous if $\liminf_{x\rightarrow x_0} f(x) \geq f(x_0)$ for all $x_0\in X$. Let $\calB(\Omega)$ denote the Borel field of some set $\Omega$, i.e., the subsets of $\Omega$ generated through relative complements and countable unions of all open subsets of $\Omega$. For the metric spaces $X$ and $Y$, a function $f:X\rightarrow Y$ is Borel measurable if for each open set $O\subseteq Y$, we have $f^{-1}(O):=\{x\in X\mid f(x)\in O\}\in\calB(X)$. For the metric spaces $X$ and $Y$, a set-valued mapping $F:X\rightrightarrows Y$ is Borel measurable if for every open set $O\subseteq Y$, we have $F^{-1}(O):=\{x\in X \mid F(x)\cap O\neq\emptyset\}\in\calB(X)$ \citep[Def. 14.1]{rockafellar:wets:2009}.

\begin{proposition}[Existence and measurability]\label{prop:exist}
    If \cref{as:convex,as:cost} hold, then for all $d\in\calD$, $V_d:\calX\times\Theta\rightarrow \bbR_{\geq 0}$ is convex and continuous, $V_d^0:\calX\rightarrow \bbR_{\geq 0}$ is convex and lower semicontinuous, $\theta^0:\calX\rightrightarrows\Theta$ is Borel measurable, and $\theta^0_d(x)\neq\emptyset$ for all $x\in\calX$.  
\end{proposition}
\begin{proof}
    We have that $L(\cdot)$ is continuous and $\bbB_d^N$ is compact \citep[Lemma A.6]{nguyen:shafieezadeh:kuhn:mohajerin:2023}. Therefore, $V_d(x,\theta)$ is continuous for all $(x,\theta)\in\bbR^n\times\Theta$ \citep[Thm. C.28]{rawlings:mayne:diehl:2020}. 
    We also have that $L(x,\theta,\bfSigma)$ is convex in $(x,\theta)\in\bbR^n\times\Theta$ for all $\bfSigma\in\bbB_d^N$. From Danskin's theorem, we have that $V_d(x,\theta)$ is also convex. Since $V_d(x,\theta)$ is continuous and $\Pi(x)$ is compact for each $x\in\calX$, we have that $\theta^0_d(x)\neq\emptyset$, i.e., the minimum is attained, for all $x\in\calX$.    
    Since $\bbZ$ and $\bbX_f$ are closed, we have that
    \begin{equation*}
    \calZ := \bigcap_{\bfw\in\bbW^N} \left\{
	(x,\bfM,\bfv)\in\bbR^n\times\Theta\ \middle| \ \begin{matrix} 
	\bfx = \bfA x + \bfB \bfv + (\bfB\bfM + \bfG)\bfw \\
	\bfu = \bfM\bfw + \bfv \\ 
	(x(k),u(k))\in\bbZ \ \forall \ k\in\bbI_{0:N-1} \\
	x(N) \in \bbX_f \end{matrix} \right\}
    \end{equation*}
    is also closed. From the Proof of \cref{lem:compact}, we have that $\calZ\subseteq\bbR^n\times(\mathbb{M}\times\bbU)$ in which $\bbU$ and $\mathbb{M}$ are compact. Therefore,
    \begin{equation*}
        V^0_d(x) = \min_{\theta}\left\{V_d(x,\theta) \mid (x,\theta)\in\calZ \right\}
    \quad \text{and} \quad 
        \theta^0_d(x) = \arg\min_{\theta}\left\{V_d(x,\theta) \mid (x,\theta)\in\calZ \right\}
    \end{equation*}
    From \citet[Prop. 7.33]{bertsekas:shreve:1996}, we have that $V_d^0:\calX\rightarrow\bbR_{\geq 0}$ is lower semicontinuous and $\theta^0:\calX\rightrightarrows\Theta$ is Borel measurable.  
\end{proof}
Since $\theta^0_d(x)$ is Borel measurable and $\bbU$ is compact, we have from \cref{prop:exist} and \citep[Lemma 7.18]{bertsekas:shreve:1996} that there exists a selection rule such that $\kappa_d:\calX\rightarrow\bbU$ is also Borel measurable. Thus, the closed-loop system $\phi(k;x,\bfw_{\infty})$ is measurable with respect to $\bfw_{\infty}$ for all $k\geq 0$. Moreover, all probabilities and expected values for the closed-loop system $\phi(\cdot)$ are well defined. 

We also have the following corollary to \cref{prop:exist} if $\bbZ$, $\bbX_f$, and $\bbW$ are polytopes. 
\begin{corollary}[Continuity of optimal value function]\label{cor:contV}
    If \cref{as:convex,as:cost} hold and $\bbZ,\bbX_f,\bbW$ are polyhedral, then for all $d\in\calD$, $V^0_d:\calX\rightarrow\bbR_{\geq 0}$ is continuous. 
\end{corollary}
\begin{proof}
    Since $\bbZ,\bbX_f,\bbW$ are polyhedral, we have from \citet[Cor. 3.8]{goulart:2007} $\calX$
    is polyhedral. From the Proof of \cref{lem:compact}, we have that $\calZ\subseteq\bbR^n\times(\mathbb{M}\times\bbU)$ in which $\bbU$ and $\mathbb{M}$ are bounded. Recall from the Proof of \cref{prop:exist} that $V_d:\bbR^n\times\Theta\rightarrow\bbR^n$ is continuous. Therefore,
    \begin{equation*}
        V^0_d(x) = \min_{\theta}\left\{V_d(x,\theta) \mid (x,\theta)\in\calZ \right\}
    \quad \text{and} \quad 
        \theta^0_d(x) = \arg\min_{\theta}\left\{V_d(x,\theta) \mid (x,\theta)\in\calZ \right\}
    \end{equation*}
    From \citet[Thm. C.34]{rawlings:mayne:diehl:2020}, we have that $V^0_d:\calX\rightarrow\bbR$ is continuous and $\theta^0_d:\calX \rightrightarrows \Theta$ is outer semicontinuous. 
\end{proof}

\section{Frank-Wolfe algorithm and step-sizes}\label{s:fw}

Frank-Wolfe algorithms that exploit the structure of the min-max program in \cref{eq:dro} have shown promising results for similar DRO problems (e.g., \citep{nguyen:shafieezadeh:kuhn:mohajerin:2023,sheriff:mohajerin:2023}). Thus, we propose such an algorithm here based on these previous algorithms. 

Analogous to the NT algorithm, we assume that $\widehat{\Sigma}\succ 0$ and for given any (fixed) $x\in\calX$ and $d\in\calD$, we use the shorthand notation for the constraint $\theta\in\Pi:=\Pi(x)$ and objective function $L_x(\cdot)=L(x,\cdot)$ in~\cref{eq:L_obj} to rewrite the min-max optimization problem in~\cref{eq:dro_L} concisely as \cref{eq:innermax}. The structure of the set~$\bbB_d^N$ in \eqref{B^N_d} allows us to expand \cref{eq:innermax}~as 
\begin{equation}\label{eq:innermax_iid_fw}
    f(\theta) = L_x(\theta,\mathbf{0}) + \sum_{k=0}^{N-1}\max_{\Sigma_k\in\bbB_d}\tr\big(Z_k(\theta)\Sigma_k\big),
\end{equation}
where the matrix $Z_k(\theta)\in\bbR^{q\times q}$ is the $k$\textsuperscript{th} block diagonal of $(H_u\bfM+H_w)'(H_u\bfM+H_w)$. Each maximization in \cref{eq:innermax_iid_fw} can be solved in finite time using a bisection algorithm detailed in \cite[Alg. 2, Thm. 6.4]{nguyen:shafieezadeh:kuhn:mohajerin:2023}. Hence, we assume to have access to
\begin{equation}\label{eq:innermax_k_fw}
    \bfSigma^*(\theta) := \arg\max_{\bfSigma\in\bbB_d^N} L_x(\theta,\bfSigma),
\end{equation}

Since the solution to \cref{eq:innermax_k_fw} is unique for $\widehat{\Sigma}\succ 0$ \citep[Prop. A.2]{nguyen:shafieezadeh:kuhn:mohajerin:2023}, we have from Danskin's theorem that $f(\theta)$ is convex and 
\begin{equation*}
    \nabla f(\theta) = \nabla_{\theta} L_x(\theta, \bfSigma^0(\theta)) 
\end{equation*}
i.e., the gradient of $f(\cdot)$ at $\theta$ is given by the gradient of $L_x(\cdot)$ with respect to $\theta$, evaluated at $(\theta,\bfSigma^0(\theta))$. 
Thus, we can define the gradient of $f(\cdot)$ as a quasi-analytic expression and the first-order oracle as
\begin{align}
\label{FW_oracle}
    F_{1}(\theta) := \arg\min_{\vartheta\in\Pi}\nabla f(\theta)'\vartheta
\end{align}
If the set $\Pi$ is a polytope, the oracle \cref{FW_oracle} is a linear program (LP). 
The solution in \cref{FW_oracle} defines the search direction for each iteration of the FW algorithm through the update rule:
\begin{equation}\label{eq:fw_update}
    \theta_{t+1} = \theta_t + \eta_t\left(F_1(\theta_t) - \theta_t\right)
\end{equation}
in which $\eta_t\in(0,1]$ is the step-size, chosen according to some (adaptive) rule that is subsequently introduced. 

For either the FW or NT algorithm, the adaptive stepsize rule is
\begin{equation}\label{eq:adaptive_fw}
    \eta_t(\beta) = \min\left\{1, \frac{(\theta_t-F_1(\theta_t))'\nabla f(\theta_t)}{\beta |\theta_t-F_1(\theta_t)|^2}\right\}
\end{equation}
in which $\beta$ is the global smoothness parameter for $f(\cdot)$, i.e., $\beta>0$ satisfies
\begin{equation*}
    |\nabla f(\theta_1) - \nabla f(\theta_2)| \leq \beta | \theta_1-\theta_2| \qquad \forall \ \theta_1,\theta_2\in\Pi
\end{equation*}
Note that we do not verify that $f(\theta)$ defined in \cref{eq:innermax} is in fact $\beta$-smooth. 

The fully adaptive stepsize rule, can improve the convergence of the FW (or NT) algorithm by replacing the global smoothness parameter $\beta$ in \cref{eq:adaptive_fw} with an adaptive smoothness parameter $\beta_t$ \citep{pedregosa:negiar:askari:jaggi:2020}. We require this $\beta_t$ to satisfy the inequality
\begin{equation}\label{eq:fadaptive}
    f\Big(\theta_t + \eta_t(\beta_t)\big(F_1(\theta_t)-\theta_t\big)\Big) \leq f(\theta_t)  + \eta_t(\beta_t)\big(F_1(\theta_t)-\theta_t\big)'\nabla f(\theta_t) + \frac{1}{2}\beta_t\eta_t(\beta_t)^2|F_1(\theta_t)-\theta_t|^2
\end{equation}
in which $\eta_t(\beta_t)$ is the adaptive stepsize calculation in \cref{eq:adaptive_fw}. The value of $\beta_t$ at each iteration is chosen according to a backtracking line search algorithm. Specifically, $\beta_t$ is chosen as the smallest element of the discrete search space $(\beta_{t-1}/\zeta)\cdot\{1,\tau,\tau^2,\tau^3,\dots\}$ that satisfies \cref{eq:fadaptive}, in which $\zeta,\tau>1$ are prescribed line search parameters. 

Unfortunately, Frank-Wolfe algorithms for MPC optimization problems are often limited to sublinear convergence rates because the constraint set $\Pi$ is not strongly convex (e.g., polytope) and the solution to the optimization problem is frequently on the boundary of $\Pi$. We observe this same limitation for DRMPC as demonstrated in \cref{fig:convergence}.

\section{Derivatives}

\begin{equation*}
    \frac{\partial L}{\partial \bfv} = 2H_u'H_x x + 2H_u'H_u\bfv
\end{equation*}

\begin{align*}
    \frac{\partial L}{\partial \bfM} & = \frac{\partial}{\partial \bfM}\tr\left(\bfM'H_u'H_u\bfM\bfSigma\right) + \frac{\partial}{\partial \bfM}2\tr\left(H_w'H_u\bfM\bfSigma\right) +  \frac{\partial}{\partial \bfM}2\tr\left(H_w'H_w\bfSigma\right)\\
    & = 2H_u'H_u\bfM\bfSigma + 2H_u'H_w\bfSigma
\end{align*}

\bibliographystyle{abbrvnat}
\bibliography{paper}

\end{document}